\documentclass[11pt]{article}
\usepackage{amsmath,amssymb,amsthm}
\usepackage{geometry}
\usepackage{stmaryrd}
 \usepackage{longtable}
\usepackage{graphicx}
\usepackage{caption}
\usepackage{booktabs}
\usepackage{hyperref}
\usepackage{url}
\usepackage{placeins}
\usepackage{float}
\newtheorem{proposition}{Proposition}
\newtheorem{corollary}{Corollary}
\newtheorem{theorem}{Theorem}
\newtheorem{lemma}{Lemma}
\newtheorem{remark}{Remark}

\newcommand{\ptop}{p_{\top}}
\newcommand{\pbot}{p_{\bot}}

\begin{document}
\hypersetup{hidelinks,pageanchor=false}
\begin{titlepage}

    \centering
    \vspace*{2.5cm}

    {\Large\bfseries Gödel Implication on Finite Chains: \\ Catalan Enumeration and a Continuum Limit Law}
    \vspace{0.5cm}
 
    {\Large\scshape Volkan Yildiz\par}
    \vspace{0.3cm}
    {\large\texttt{vo1kan@hotmail.co.uk}\par}
    \vfill
    {\normalsize \today \par}

\begin{abstract}
We study the distribution of output values produced by all fully bracketed implication terms on $n$ distinct variables in G\"odel $m$-valued logic on the finite chain $V_m$. The Catalan decomposition of bracketings yields a stratified system of recurrences for the output counts. Passing to tail sums reduces this system to a finite iteration of a single map $\Phi_x$ and its inverse branch $\Psi_x$. For each fixed $m$, this reduction gives a common dominant square-root singularity at $r=1/(4m)$, coefficient asymptotics of order $(4m)^n n^{-3/2}$, and limiting proportions for every truth level. In the four-valued case these proportions are obtained explicitly in nested radicals.

The main new phenomenon appears when the number of truth values also tends to infinity. If $K_m$ denotes the output level under the limiting $n\to\infty$ distribution and $X_m=K_m/(m-1)$, then
\[
X_m\Rightarrow X,
\]
where $X$ has density $(1+2t)^{-3/2}$ on $[0,1)$ together with an atom of mass $1/\sqrt3$ at $1$. In particular, the top mass converges to $1/\sqrt3$, while the bottom mass is asymptotic to $1/m$. We further show that the same probability law arises intrinsically when G\"odel implication is defined directly on $[0,1]$ and valuations are sampled from Lebesgue product measure: the continuum tail generating function satisfies an exact first-order differential equation whose critical specialization is the ODE appearing in the finite-chain Euler limit. We also record a root-split refinement whose pair-count generating functions factor as products of the level generating functions and inherit the same square-root scale.
\vfill

\end{abstract}

\end{titlepage}

\makeatletter
\newcommand{\oneLinePart}[1]{
  \cleardoublepage
  \refstepcounter{part}
  \thispagestyle{plain}
  \addcontentsline{toc}{part}{\partname~\thepart\space #1}
  \begin{center}
    \makebox[\textwidth][c]{\Large\bfseries \partname~\thepart\quad #1}
  \end{center}
  \vspace{1.2em}
}
\makeatother

\newpage

\setcounter{tocdepth}{1}
\makeatletter
\renewcommand*{\l@section}{\@dottedtocline{1}{1.5em}{3.2em}}
\makeatother
\tableofcontents
\bigskip

\newpage
\section{Introduction}

Fully bracketed words are among the most elementary Catalan objects, but a fixed binary operation can turn their common syntax into markedly different semantic behaviour. For the implication word
\[
p_1\Rightarrow p_2\Rightarrow\cdots\Rightarrow p_n,
\]
there are exactly $C_{n-1}$ full bracketings. Once a truth-value algebra is fixed, every bracketing determines a term operation, and one may ask not only whether different bracketings induce different operations, but also how their output values are distributed when all valuations are taken into account.

The first question belongs to the broader study of associative spectra of binary operations, where one counts the distinct term operations induced by bracketings; see, for example, Cs\'ak\'any and Waldhauser~\cite{CsakanyWaldhauser2000} and Liebscher and Waldhauser~\cite{LiebscherWaldhauser2009}. The present paper studies a different, complementary statistic. We retain every Catalan bracketing, evaluate it on every valuation, and count the total number of occurrences of each truth value. In earlier work~\cite{CameronYildiz2010,Yildiz2011b,Yildiz2012}, this programme was developed for classical implication, and later for Kleene's three-valued logic~\cite{Yildiz2020}. Here we turn to G\"odel implication on finite chains.

Let
\[
V_m=\{v_0<v_1<\cdots<v_{m-1}\},\qquad v_0=0,\quad v_{m-1}=1,
\]
and define
\[
v_i\Rightarrow_G v_j=
\begin{cases}
v_{m-1},&i\le j,\\
v_j,&i>j.
\end{cases}
\]
The order-theoretic form of this operation creates a useful hierarchy. A non-top value $v_j$ can occur at the root only when the right subformula has value $v_j$ and the left subformula has a strictly larger value. The resulting recurrence system is therefore naturally triangular when written in terms of upper tail sums.

This observation is the main finite-chain mechanism of the paper. If $G_j(x)$ counts output occurrences of $v_j$ and
\[
H_k(x)=\sum_{j=k}^{m-1}G_j(x),
\]
then the Catalan decomposition gives
\[
H_k(x)=\Phi_x\bigl(H_{k+1}(x)\bigr),
\qquad
\Phi_x(u)=u+\frac{x}{1-u}.
\]
The inverse branch
\[
\Psi_x(v)=\frac{1+v-\sqrt{(1-v)^2+4x}}{2}
\]
therefore reconstructs the entire distribution by a finite iteration of depth $m-1$. This converts an $m$-component recurrence system into a one-dimensional analytic chain.

The paper has four principal conclusions. First, for every fixed $m$, all truth-level generating functions have the same dominant singularity $r=1/(4m)$ and the same universal Catalan scale $(4m)^n n^{-3/2}$; only the leading constants depend on the truth level. Second, in the four-valued case the limiting proportions are obtained explicitly in nested radicals. Third, and most importantly, the family of limiting output distributions has a non-trivial continuum limit as $m\to\infty$. If $K_m$ is distributed according to the limiting truth-level proportions and
\[
X_m=\frac{K_m}{m-1},
\]
then
\[
X_m\Rightarrow X,
\qquad
\Pr(X\ge t)=\frac{1}{\sqrt{1+2t}}\quad(0\le t<1),
\]
with an atom of mass $1/\sqrt3$ at $t=1$. Equivalently, the absolutely continuous part has density $(1+2t)^{-3/2}$ on $[0,1)$. In particular,
\[
\ptop(m)\longrightarrow\frac1{\sqrt3},
\qquad
\pbot(m)\sim\frac1m.
\]
Fourth, a root-split refinement factors into products $G_i(x)G_j(x)$ and consequently inherits the same square-root asymptotics. We treat this final refinement as a structural corollary of the main generating-function framework rather than as a separate asymptotic theory.

The continuum limit is the main qualitative result. Its proof rests on the same inverse iteration that controls fixed $m$: after writing $u_k=1-w_k$ along the critical $\Psi$-chain, the discrete recursion becomes an Euler scheme for
\[
u'(t)=\frac{1}{4u(t)},\qquad u(0)=\frac12.
\]
The solution $u(t)^2=\frac14+\frac t2$ converts the derivative product governing the tail masses into a Riemann integral, from which the limiting survival function follows.

As a direct check on the meaning of this limit, Part~III also treats G\"odel implication on the standard interval $[0,1]$ equipped with Lebesgue product measure. In that model the analogue of the tail-sum calculus is an exact differential equation rather than an asymptotic Euler scheme. At its dominant Catalan singularity the equation reduces precisely to the ODE above, and its large-$n$ output law is again the measure in Theorem~\ref{thm:limit-law-shape}.

\medskip
\noindent\textbf{Notation and organisation.}
Lower-case symbols such as $g_n$ and $g_n^{(j)}$ denote coefficient sequences, while upper-case symbols such as $G(x)$ and $G_j(x)$ denote their ordinary generating functions. The symbols $H_k$ denote tail-sum OGFs, $p_k(m)$ the limiting point masses as $n\to\infty$, and $q_k(m)=\sum_{j\ge k}p_j(m)$ their tails. Critical values of the inverse iteration are written $w_k$, with $u_k=1-w_k$ in the large-$m$ analysis. The rescaled output level and its weak limit are denoted by $X_m$ and $X$, respectively; this keeps them distinct from the four-valued top-output OGF $T(x)$. In Part~IV, $\rho_k$ and $\sigma_k$ denote the regular and square-root coefficients of $G_k$. Part~I develops the four-valued model and the analytic propagation lemma; Part~II treats general fixed $m$; Part~III proves the continuum limit law and then recovers the same law directly from Lebesgue semantics on $[0,1]$; Part~IV records the root-split refinement. Numerical tables are included as checks and illustrations, not as substitutes for the asymptotic arguments.

\makeatletter
\@addtoreset{section}{part}
\@addtoreset{subsection}{part}
\makeatother

\renewcommand{\thesection}{\Roman{part}.\arabic{section}}
\renewcommand{\thesubsection}{\thesection.\arabic{subsection}}
\renewcommand{\theHsection}{\arabic{part}.\arabic{section}}
\renewcommand{\theHsubsection}{\theHsection.\arabic{subsection}}

\newpage

\oneLinePart{Gödel Four--Valued Implication}

\section{Recurrence Relations}
Let $p_1, p_2, \ldots ,p_n$ be distinct propositional variables.
Let \(V_4=\{0<a<b<1\}\) and consider the formula \(p_1 \Rightarrow_G p_2\).
Gödel implication is defined by
\[
x \Rightarrow_G y =
\begin{cases}
1, & \text{if } x \le y,\\
y, & \text{if } x > y .
\end{cases}
\]

\medskip
The induced operation table is shown in Table~\ref{tab:godel-op}.

\begin{table}[h]
\centering
\[
\begin{array}{c|cccc}
p_1 \Rightarrow_G p_2 & 0 & a & b & 1\\
\hline
0 & 1 & 1 & 1 & 1\\
a & 0 & 1 & 1 & 1\\
b & 0 & a & 1 & 1\\
1 & 0 & a & b & 1
\end{array}
\]
\caption{Gödel implication on \(V_4=\{0<a<b<1\}\).}
\label{tab:godel-op}
\end{table}

The full truth table, listing all valuations of \(p_1\) and \(p_2\), is given in the appendix;
see Table~\ref{tab:godel-n2}. Truth tables for all valuations of \(p_1\), \(p_2\), and \(p_3\)
are likewise provided in the appendix; see Table~\ref{tab:godel4-n3-sidebyside}.

In the Gödel four--valued chain \(V_4=\{0<a<b<1\}\) we write
\[
f_n=\#\text{ of zeros},\qquad a_n=\#\text{ of \(a\)'s},\qquad b_n=\#\text{ of \(b\)'s},\qquad
t_n=\#\text{ of ones},
\]
counted across the output columns of the truth tables of all fully bracketed implications
on \(n\) variables. Thus
\[
g_n=f_n+a_n+b_n+t_n,
\qquad g_n=4^n C_{n-1}.
\]

By the Catalan decomposition of bracketings, we have
\[
g_n=\sum_{i=1}^{n-1} g_i g_{n-i}
   =\sum_{i=1}^{n-1}(f_i+a_i+b_i+t_i)(f_{n-i}+a_{n-i}+b_{n-i}+t_{n-i}).
\]

Using the Gödel implication
\[
x\Rightarrow_G y=
\begin{cases}
1,& x\le y,\\
y,& x>y,
\end{cases}
\]
all outputs other than 1 come only from the following product contributions.

\medskip

\noindent\textbf{Output $b$:}
\[
b_n=\sum_{i=1}^{n-1}
\underbrace{t_i\,b_{n-i}}_{\text{$1\Rightarrow b=b$}}.
\]

\noindent\textbf{Output $a$:}
\[
a_n=\sum_{i=1}^{n-1}\left(
\underbrace{b_i\,a_{n-i}}_{\text{$b\Rightarrow a=a$}}
+
\underbrace{t_i\,a_{n-i}}_{\text{$1\Rightarrow a=a$}}
\right)
=\sum_{i=1}^{n-1}\underbrace{(b_i+t_i)\,a_{n-i}}_{\text{output $a$}}.
\]

\noindent\textbf{Output $0$:}
\[
f_n=\sum_{i=1}^{n-1}\left(
\underbrace{a_i\,f_{n-i}}_{\text{$a\Rightarrow 0=0$}}
+
\underbrace{b_i\,f_{n-i}}_{\text{$b\Rightarrow 0=0$}}
+
\underbrace{t_i\,f_{n-i}}_{\text{$1\Rightarrow 0=0$}}
\right)
=\sum_{i=1}^{n-1}\underbrace{(a_i+b_i+t_i)\,f_{n-i}}_{\text{output $0$}}.
\]

\noindent\textbf{Output $1$:}
\[
t_n=g_n-f_n-a_n-b_n.
\]

Equivalently,
\[
t_n=\sum_{i=1}^{n-1}\Big(
g_i g_{n-i}
-
t_i b_{n-i}
-
(b_i+t_i)a_{n-i}
-
(a_i+b_i+t_i)f_{n-i}
\Big).
\]

\medskip

We define the ordinary generating functions
\[
F(x)=\sum_{n\ge1} f_n x^n,\qquad
A(x)=\sum_{n\ge1} a_n x^n,\qquad
B(x)=\sum_{n\ge1} b_n x^n,\qquad
T(x)=\sum_{n\ge1} t_n x^n,
\]
and
\[
G(x)=\sum_{n\ge1} g_n x^n.
\]

Since $g_n=4^n C_{n-1}$, we have
\[
G(x)=\sum_{n\ge 1}4^n C_{n-1}x^n=\frac{1-\sqrt{1-16x}}{2}.
\]

Our OGFs start at $n=1$ while the Catalan recurrences hold for $n\ge 2$, so the $n=1$
terms must be retained. For example,
$b_n=\sum_{i=1}^{n-1} t_i b_{n-i}$ gives $B(x)-x=T(x)B(x)$, i.e.\ $B(x)=x+T(x)B(x)$.
Similarly,
\[
A(x)=x+(B(x)+T(x))A(x),\qquad
F(x)=x+(A(x)+B(x)+T(x))F(x),
\]
and
\[
T(x)=G(x)-F(x)-A(x)-B(x).
\]

\subsection{Component sequences: the four--valued case}
Recall that
\[
g_n=f_n+a_n+b_n+t_n.
\]

The total sequence satisfies
\[
g_n=4^n C_{n-1}\sim \frac{16^n}{4\sqrt{\pi}\,n^{3/2}}.
\]
Numerical evaluation of the recurrences shows that each component grows on the
same exponential scale as \(g_n\), and the ratios \(f_n/g_n\), \(a_n/g_n\), \(b_n/g_n\),
\(t_n/g_n\) rapidly stabilise. For \(n=20\) one obtains approximately
\[
\frac{f_n}{g_n}\approx 0.1498,\qquad
\frac{a_n}{g_n}\approx 0.1006,\qquad
\frac{b_n}{g_n}\approx 0.0729,\qquad
\frac{t_n}{g_n}\approx 0.6767.
\]
Thus, in the four-valued case, the value \(1\) dominates asymptotically,
while the three lower truth values occur with smaller but non--negligible limiting
frequencies.

\section{\texorpdfstring{The $\Phi$--reduction and the governing generating function}{The Phi-reduction and the governing generating function}}

Consider the Gödel four--valued chain \(V_4=\{0<a<b<1\}\). Let
\[
F(x)=\sum_{n\ge1} f_n x^n,\qquad
A(x)=\sum_{n\ge1} a_n x^n,\qquad
B(x)=\sum_{n\ge1} b_n x^n,\qquad
T(x)=\sum_{n\ge1} t_n x^n
\]
denote the component generating functions corresponding to the output values
\(0,a,b,1\), respectively. Set
\[
G(x):=F(x)+A(x)+B(x)+T(x)
=\sum_{n\ge1} g_n x^n,
\qquad
g_n=4^n C_{n-1},
\]
so that
\begin{equation}\label{eq:G-closed}
G(x)=\sum_{n\ge1}4^nC_{n-1}x^n=\frac{1-\sqrt{1-16x}}{2}.
\end{equation}

\medskip

\noindent
We now re-introduce the functional equations (including the \(n=1\) terms).
The Gödel--\(4\) Catalan decompositions give:
\begin{equation}\label{eq:GF-system-godel4}
B(x)=x+T(x)B(x),\qquad
A(x)=x+(B(x)+T(x))A(x),\qquad
F(x)=x+(A(x)+B(x)+T(x))F(x),
\end{equation}
together with
\begin{equation}\label{eq:T-from-total}
T(x)=G(x)-F(x)-A(x)-B(x).
\end{equation}
\begin{proposition}[Recovering the remaining generating functions]\label{prop:recover-godel4}
Once \(T(x)\) is known, the remaining component generating functions are
\begin{equation}\label{eq:recover-ABF}
B(x)=\frac{x}{1-T(x)},\qquad
A(x)=\frac{x}{1-(B(x)+T(x))},\qquad
F(x)=\frac{x}{1-(A(x)+B(x)+T(x))}.
\end{equation}
\end{proposition}

\begin{proof}
These formulas follow directly by rearranging \eqref{eq:GF-system-godel4}.
\end{proof}

\begin{proposition}[$\Phi$--reduction: one equation in one unknown]\label{prop:phi-reduction-godel4}
Define the tail sums
\[
H_3(x):=T(x),\qquad
H_2(x):=B(x)+T(x),\qquad\]
\[
H_1(x):=A(x)+B(x)+T(x),\qquad
H_0(x):=F(x)+A(x)+B(x)+T(x)=G(x),
\]
and define the map
\begin{equation}\label{eq:Phi-def}
\Phi_x(u):=u+\frac{x}{1-u}.
\end{equation}
Then
\[
H_2(x)=\Phi_x(H_3(x)),\qquad
H_1(x)=\Phi_x(H_2(x)),\qquad
H_0(x)=\Phi_x(H_1(x)).
\]
Equivalently,
\begin{equation}\label{eq:G-Phi3-T}
\,G(x)=\Phi_x^{\,3}\!\big(T(x)\big).\,
\end{equation}
\end{proposition}

\begin{proof}
From \eqref{eq:GF-system-godel4} we obtain
\[
B(x)=\frac{x}{1-T(x)}.
\]
Hence
\[
H_2(x)=B(x)+T(x)=T(x)+\frac{x}{1-T(x)}=\Phi_x(H_3(x)).
\]
and, from \eqref{eq:GF-system-godel4},
\[
A(x)=\frac{x}{1-(B(x)+T(x))}=\frac{x}{1-H_2(x)},
\]
so
\[
H_1(x)=A(x)+H_2(x)=H_2(x)+\frac{x}{1-H_2(x)}=\Phi_x(H_2(x)).
\]
Finally,
\[
F(x)=\frac{x}{1-(A(x)+B(x)+T(x))}=\frac{x}{1-H_1(x)},
\]
so
\[
H_0(x)=F(x)+H_1(x)=H_1(x)+\frac{x}{1-H_1(x)}=\Phi_x(H_1(x)).
\]
Iterating gives \(H_0=\Phi_x^{3}(H_3)\), i.e. \eqref{eq:G-Phi3-T}.
\end{proof}

\medskip 

\begin{proposition}[Explicit inverse and nested radicals]\label{prop:Psi-inverse-godel4}
Viewed as a map on $x\mathbb{C}[[x]]$, \(\Phi_x\) in \eqref{eq:Phi-def} is invertible. Its inverse branch is
\begin{equation}\label{eq:Psi-def}
\Psi_x(v):=\frac{1+v-\sqrt{(1-v)^2+4x}}{2},
\end{equation}
where the square root is the formal branch with constant term $1$. Thus, whenever
$v(x)\in x\mathbb{C}[[x]]$, one has $\Psi_x(v(x))\in x\mathbb{C}[[x]]$ and
\(\Phi_x(\Psi_x(v))=v\).
Thus,
\begin{equation}\label{eq:T-Psi3-G}
\,T(x)=\Psi_x^{\,3}\!\big(G(x)\big),\qquad
G(x)=\frac{1-\sqrt{1-16x}}{2}.\,
\end{equation}
\end{proposition}

\begin{proof}
If \(v=\Phi_x(u)=u+\frac{x}{1-u}\), then \(v(1-u)=u(1-u)+x\), i.e.
\(u^2-(1+v)u+(v-x)=0\). Solving this quadratic gives
\[
u=\frac{1+v\pm\sqrt{(1-v)^2+4x}}{2}.
\]
For $v(x)\in x\mathbb{C}[[x]]$, the branch with zero constant term is obtained by choosing
the minus sign and the square root with constant term $1$, giving \eqref{eq:Psi-def}.
Applying \(\Psi_x\) three times to \eqref{eq:G-Phi3-T} gives \eqref{eq:T-Psi3-G}.
\end{proof}

\medskip


\section{Asymptotics}

Before analysing the four-valued model, we isolate the analytic step used throughout the paper.

\subsection{A square-root propagation principle}

\begin{theorem}[Propagation under analytic iteration]\label{thm:blackbox-sqrt}
Let $r>0$, and let $\mathcal D$ be a $\Delta$--domain at $r$. Suppose
\[
H_0(x)=a_0-c_0\sqrt{1-\frac{x}{r}}
+O\!\left(1-\frac{x}{r}\right),
\qquad c_0\ne0,
\]
as $x\to r$ in $\mathcal D$. Let $\Psi(x,v)$ be analytic near the finite critical chain
\[
a_{k+1}=\Psi(r,a_k),\qquad 0\le k\le K-1,
\]
and define $H_{k+1}(x)=\Psi(x,H_k(x))$. Assume the compositions are analytic in $\mathcal D$ and
$\partial_v\Psi(r,a_k)\ne0$ along the chain. Then
\[
H_k(x)=a_k-c_k\sqrt{1-\frac{x}{r}}
+O\!\left(1-\frac{x}{r}\right),
\]
where
\[
c_k=c_0\prod_{j=0}^{k-1}\partial_v\Psi(r,a_j).
\]
Consequently,
\[
[x^n]H_k(x)\sim \frac{c_k}{2\sqrt\pi}\,r^{-n}n^{-3/2}.
\]
\end{theorem}

\begin{proof}
Put $t=1-x/r$. If $H_k(x)=a_k-c_k\sqrt t+O(t)$, analyticity of $\Psi$ gives
\[
\Psi(x,H_k(x))
=\Psi(r,a_k)+\partial_v\Psi(r,a_k)(H_k(x)-a_k)+O(t),
\]
because $x-r=O(t)$ and $(H_k(x)-a_k)^2=O(t)$. Hence
\[
H_{k+1}(x)
=a_{k+1}-c_k\partial_v\Psi(r,a_k)\sqrt t+O(t).
\]
Iteration gives the product formula for $c_k$, and the coefficient asymptotic follows from the standard square-root transfer theorem~\cite[Ch.~VI]{FlajoletSedgewick2009}.
\end{proof}

The remaining issue is to verify the required $\Delta$--analyticity for the G\"odel tail chain. The following lemma does this once for all finite $m$.

\begin{lemma}[No secondary branch points on the dominant circle]\label{lem:delta-domain-godel}
Fix $m\ge2$, put $r=1/(4m)$, and let
\[
H_0(x)=\frac{1-\sqrt{1-4mx}}2,
\qquad
H_{\ell+1}(x)=\Psi_x(H_\ell(x)),
\quad 0\le \ell\le m-2.
\]
Then the relevant branches of all $H_\ell$ extend analytically to a common $\Delta$--domain at $r$. In particular, no additional branch point of the inverse iteration lies on $|z|=r$.
\end{lemma}

\begin{proof}
Every $H_\ell$ has nonnegative coefficients, hence for $|z|\le r$,
\[
|H_\ell(z)|\le H_\ell(r)\le H_0(r)=\frac12,
\qquad
|1-H_\ell(z)|\ge\frac12.
\]
From $H_\ell=\Phi_z(H_{\ell+1})$ we obtain
\[
(1-H_\ell(z))^2+4z
=\frac{\bigl((1-H_{\ell+1}(z))^2+z\bigr)^2}
       {(1-H_{\ell+1}(z))^2}.
\]
Moreover,
\[
|(1-H_{\ell+1}(z))^2|\ge\frac14>r\ge |z|,
\]
so the numerator on the right cannot vanish. Thus the radicand defining the next inverse step is nonzero on the closed dominant disk. The series $H_\ell$ converges absolutely and uniformly on $|z|\le r$, because its coefficients are nonnegative and $H_\ell(r)<\infty$; hence the relevant denominators and radicands are continuous there.

Fix a small $\varepsilon>0$ and let
\[
K_\varepsilon:=\{z:|z|=r,\ |z-r|\ge\varepsilon\}.
\]
We continue the inverse chain across $K_\varepsilon$ inductively. The function $H_0(z)=(1-\sqrt{1-4mz})/2$ is analytic in a neighbourhood of $K_\varepsilon$. Suppose that $H_\ell$ has already been continued analytically to a neighbourhood of this compact arc. On $K_\varepsilon$ the radicand $(1-H_\ell(z))^2+4z$ is nonzero by the estimate above; hence its modulus has a positive minimum there. After shrinking the neighbourhood if necessary, the radicand remains nonzero, so it has a local analytic square-root branch agreeing with the branch chosen inside $|z|<r$. Therefore $H_{\ell+1}=\Psi_z(H_\ell)$ also extends analytically across $K_\varepsilon$. Because there are only finitely many inverse steps, intersecting the resulting neighbourhoods gives a common open neighbourhood of $K_\varepsilon$ for all $H_\ell$.

Near $r$ itself, the inverse-step radicands are nonzero at the critical values $H_\ell(r)$, so the same local branch argument gives analytic continuation of the inverse iterates in a slit neighbourhood of $r$. Combining this slit neighbourhood with the common neighbourhood of $K_\varepsilon$ yields a $\Delta$--domain at $r$ on which all inverse iterates are analytic.
\end{proof}

\subsection{The four-valued singularity and limiting top mass}

Let $r=1/16$ be the dominant singularity of $G(x)$. From \eqref{eq:G-closed},
\begin{equation}\label{eq:G-local}
G(x)=\frac12-\frac12\sqrt{1-\frac{x}{r}}
+O\!\left(1-\frac{x}{r}\right).
\end{equation}

\begin{theorem}[Limiting proportion of $1$'s and asymptotics]\label{thm:limit-ones-godel4}
Let $t_n=[x^n]T(x)$ and $g_n=[x^n]G(x)=4^nC_{n-1}$. Then
\[
\ptop:=\lim_{n\to\infty}\frac{t_n}{g_n}
\]
exists, and
\[
t_n\sim \frac{\ptop}{4}\frac{16^n}{\sqrt\pi\,n^{3/2}}.
\]
\end{theorem}

\begin{proof}
By Lemma~\ref{lem:delta-domain-godel}, the three inverse iterates defining
$T(x)=\Psi_x^{\,3}(G(x))$ satisfy the analytic hypotheses of Theorem~\ref{thm:blackbox-sqrt}. Hence
\[
T(x)=T(r)-\lambda\sqrt{1-\frac{x}{r}}
+O\!\left(1-\frac{x}{r}\right)
\]
for some $\lambda>0$. Square-root transfer gives
\[
g_n\sim\frac{16^n}{4\sqrt\pi\,n^{3/2}},
\qquad
t_n\sim\frac{\lambda}{2\sqrt\pi}\frac{16^n}{n^{3/2}}.
\]
Therefore $t_n/g_n\to2\lambda=: \ptop$, which is equivalent to the stated asymptotic.
\end{proof}

\subsection{Gödel--$4$: exact limiting proportions and asymptotic constants}

Recall
\[
G(x)=F(x)+A(x)+B(x)+T(x)=\frac{1-\sqrt{1-16x}}{2}.
\]
Set \(r=\frac1{16}\) and recall the inverse map
\[
\Psi_x(v)=\frac{1+v-\sqrt{(1-v)^2+4x}}{2},
\]
so that \(T(x)=\Psi_x^{\,3}(G(x))\) (Proposition~\ref{prop:Psi-inverse-godel4}).  At the
dominant point \(x=r\) we have \(4r=\frac14\) and thus
\begin{equation}\label{eq:Psi-at-r}
\Psi_r(v)=\frac{1+v-\sqrt{(1-v)^2+\frac14}}{2}.
\end{equation}

\medskip
\noindent\textbf{Auxiliary radicals.}
Define
\[
\beta:=\sqrt{7+2\sqrt2},\qquad
\alpha:=1+\sqrt2+\beta,\qquad
\gamma:=\sqrt{16+\alpha^2}.
\]

\begin{lemma}[Explicit evaluation of the iterates at \(x=r\)]\label{lem:iterates-at-r}
Let \(w_0:=G(r)=\frac12\) and \(w_{k+1}:=\Psi_r(w_k)\) for \(k=0,1,2\). Then
\[
w_0=\frac12,\qquad
w_1=\frac34-\frac{\sqrt2}{4},\qquad
w_2=\frac78-\frac{\sqrt2}{8}-\frac{\beta}{8},\qquad
w_3=\frac{15}{16}-\frac{\sqrt2}{16}-\frac{\beta}{16}-\frac{\gamma}{16},
\]
where \(w_3=T(r)\).
\end{lemma}

\begin{proof}
We compute successively using \eqref{eq:Psi-at-r}.

\smallskip\noindent
\emph{Step 1.} Since \(w_0=\frac12\),
\[
\sqrt{(1-w_0)^2+\frac14}=\sqrt{\left(\frac12\right)^2+\frac14}=\sqrt{\frac12}=\frac{\sqrt2}{2},
\]
hence
\[
w_1=\Psi_r\!\left(\frac12\right)
=\frac{1+\frac12-\frac{\sqrt2}{2}}{2}
=\frac34-\frac{\sqrt2}{4}.
\]

\smallskip\noindent
\emph{Step 2.} Here \(1-w_1=\frac{1+\sqrt2}{4}\), so
\[
(1-w_1)^2+\frac14=\frac{(1+\sqrt2)^2}{16}+\frac14=\frac{7+2\sqrt2}{16},
\qquad
\sqrt{(1-w_1)^2+\frac14}=\frac{\beta}{4}.
\]
Substituting into \eqref{eq:Psi-at-r} yields
\[
w_2=\Psi_r(w_1)=\frac{1+w_1-\frac{\beta}{4}}{2}
=\frac78-\frac{\sqrt2}{8}-\frac{\beta}{8}.
\]

\smallskip\noindent
\emph{Step 3.} Next \(1-w_2=\frac{1+\sqrt2+\beta}{8}=\frac{\alpha}{8}\), hence
\[
(1-w_2)^2+\frac14=\frac{\alpha^2}{64}+\frac14=\frac{\alpha^2+16}{64},
\qquad
\sqrt{(1-w_2)^2+\frac14}=\frac{\gamma}{8}.
\]
Substituting into \eqref{eq:Psi-at-r} gives
\[
w_3=\Psi_r(w_2)
=\frac{1+w_2-\frac{\gamma}{8}}{2}
=\frac{15}{16}-\frac{\sqrt2}{16}-\frac{\beta}{16}-\frac{\gamma}{16}.
\]
\end{proof}

\begin{proposition}[Derivative product and the constant \(\ptop \)]\label{prop:p1-product}
Differentiating \(\Psi_x(v)\) gives
\begin{equation}\label{eq:dPsi}
\partial_v\Psi_x(v)
=\frac12\left(1+\frac{1-v}{\sqrt{(1-v)^2+4x}}\right).
\end{equation}
Evaluating at \(x=r\) along \(w_0,w_1,w_2\) yields
\[
\partial_v\Psi_r(w_0)=\frac{2+\sqrt2}{4},\qquad
\partial_v\Psi_r(w_1)=\frac{\alpha}{2\beta},\qquad
\partial_v\Psi_r(w_2)=\frac{\alpha+\gamma}{2\gamma}.
\]
Consequently,
\begin{equation}\label{eq:p1-product}
\ptop :=\lim_{n\to\infty}\frac{t_n}{g_n}
=\prod_{k=0}^{2}\partial_v\Psi_r(w_k)
=\frac{(2+\sqrt2)\,\alpha(\alpha+\gamma)}{16\,\beta\,\gamma}.
\end{equation}
\end{proposition}
\begin{proof}
A direct differentiation and substitution using Lemma~\ref{lem:iterates-at-r}.
\end{proof}

\begin{theorem}[Exact limiting proportions for Gödel--$4$]\label{thm:godel4-exact-proportions}
The limits
\[
\pbot  :=\lim_{n\to\infty}\frac{f_n}{g_n},\qquad
p_a:=\lim_{n\to\infty}\frac{a_n}{g_n},\qquad
p_b:=\lim_{n\to\infty}\frac{b_n}{g_n},\qquad
\ptop :=\lim_{n\to\infty}\frac{t_n}{g_n}
\]
exist and satisfy \(\pbot +p_a+p_b+\ptop =1\). Moreover, they are given exactly by
\[
\pbot=\frac12-\frac{\sqrt2}{4},
\qquad
p_a=\frac{2+\sqrt2}{2\alpha\beta},
\qquad
p_b=\frac{(2+\sqrt2)\alpha}{\beta\,\gamma(\alpha+\gamma)},
\qquad
\ptop =\frac{(2+\sqrt2)\alpha(\alpha+\gamma)}{16\,\beta\,\gamma}.
\]
In particular,
\[
\ptop \approx 0.684122210733017786\ldots\, .
\]
\end{theorem}

\begin{proof}
Write
\[
d_0:=\partial_v\Psi_r(w_0),\qquad
d_1:=\partial_v\Psi_r(w_1),\qquad
d_2:=\partial_v\Psi_r(w_2).
\]
The tail sums satisfy
\[
H_0=G,\qquad H_1=A+B+T,\qquad H_2=B+T,\qquad H_3=T,
\]
and $H_{k+1}=\Psi_x(H_k)$. Theorem~\ref{thm:blackbox-sqrt} therefore shows that the limiting coefficient ratios of the tails are
\[
q_0=1,\qquad q_1=d_0,\qquad q_2=d_0d_1,\qquad q_3=d_0d_1d_2.
\]
Since the individual components are successive tail differences,
\[
\pbot=q_0-q_1=1-d_0,
\]
\[
p_a=q_1-q_2=d_0(1-d_1),
\]
\[
p_b=q_2-q_3=d_0d_1(1-d_2),
\qquad
\ptop=q_3=d_0d_1d_2.
\]
Proposition~\ref{prop:p1-product} gives
\[
d_0=\frac{2+\sqrt2}{4},\qquad
d_1=\frac{\alpha}{2\beta},\qquad
d_2=\frac{\alpha+\gamma}{2\gamma}.
\]
Substitution gives the four limiting proportions directly. For the two nontrivial simplifications, note first that
\[
\bigl(\beta-(1+\sqrt2)\bigr)\bigl(\beta+(1+\sqrt2)\bigr)
=\beta^2-(1+\sqrt2)^2=4.
\]
Since $\beta+(1+\sqrt2)=\alpha$, this gives
\[
\beta-1-\sqrt2=\frac{4}{\alpha}.
\]
Also, by definition of $\gamma$,
\[
\gamma^2-\alpha^2=16,
\qquad
\gamma-\alpha=\frac{16}{\gamma+\alpha}.
\]
Hence
\[
1-d_1=1-\frac{\alpha}{2\beta}
=\frac{\beta-1-\sqrt2}{2\beta}
=\frac{2}{\alpha\beta},
\]
and
\[
1-d_2=1-\frac{\alpha+\gamma}{2\gamma}
=\frac{\gamma-\alpha}{2\gamma}
=\frac{8}{\gamma(\alpha+\gamma)}.
\]
Together with $d_0=(2+\sqrt2)/4$, the tail differences therefore become
\[
\pbot=1-d_0=\frac12-\frac{\sqrt2}{4},
\]
\[
p_a=d_0(1-d_1)=\frac{2+\sqrt2}{2\alpha\beta},\qquad
p_b=d_0d_1(1-d_2)=\frac{(2+\sqrt2)\alpha}{\beta\gamma(\alpha+\gamma)},
\]
while $\ptop=d_0d_1d_2$ is exactly the product in \eqref{eq:p1-product}. The identity
$\pbot+p_a+p_b+\ptop=1$ is automatic from the telescoping tail differences.
\end{proof}

\begin{corollary}[False rows and non--true rows]\label{cor:false-rows-godel4}
Define \(u_n:=f_n\) and \(s_n:=g_n-t_n=f_n+a_n+b_n\). Then
\[
\lim_{n\to\infty}\frac{u_n}{g_n}=\pbot,
\qquad
\lim_{n\to\infty}\frac{s_n}{g_n}=1-\ptop=\pbot+p_a+p_b,
\]
so the limiting proportion of non--true rows is
\[
1-\ptop \approx 0.315877789266982214\ldots\, .
\]
Equivalently,
\[
S(x):=\sum_{n\ge1}s_nx^n=G(x)-T(x).
\]
\end{corollary}

\begin{theorem}[Exact asymptotic constants]\label{thm:godel4-exact-asymptotics}
As \(n\to\infty\),
\[
g_n\sim \frac{16^n}{4\sqrt{\pi}\,n^{3/2}},
\]
and each component satisfies
\[
f_n\sim \frac{\pbot}{4}\,\frac{16^n}{\sqrt{\pi}\,n^{3/2}},\qquad
a_n\sim \frac{p_a}{4}\,\frac{16^n}{\sqrt{\pi}\,n^{3/2}},\qquad
b_n\sim \frac{p_b}{4}\,\frac{16^n}{\sqrt{\pi}\,n^{3/2}},\qquad
t_n\sim \frac{\ptop}{4}\,\frac{16^n}{\sqrt{\pi}\,n^{3/2}},
\]
where \(\pbot, p_a,p_b,\ptop \) are as in Theorem~\ref{thm:godel4-exact-proportions}.
\end{theorem}

\begin{proof}
The asymptotic for \(g_n\) follows from \eqref{eq:G-local} by the transfer theorem, \cite[Ch.~VI]{FlajoletSedgewick2009}.
Each component has the same dominant square--root singularity at \(x=r\), and the limiting
ratios \(f_n/g_n\to \pbot \), \(a_n/g_n\to p_a\), \(b_n/g_n\to p_b\), \(t_n/g_n\to \ptop \) hold by
Theorem~\ref{thm:godel4-exact-proportions}. The stated component asymptotics then follow.
\end{proof}

\newpage
\begin{table}[t]
\centering
\small
\renewcommand{\arraystretch}{1.05}
\begin{tabular}{rcccc}
\hline
\(n\) & \(f_n/g_n\) & \(a_n/g_n\) & \(b_n/g_n\) & \(t_n/g_n\) \\
\hline
10  & 0.153354 & 0.103186 & 0.074472 & 0.668988 \\
20  & 0.149830 & 0.100597 & 0.072869 & 0.676705 \\
30  & 0.148687 & 0.099760 & 0.072343 & 0.679210 \\
40  & 0.148121 & 0.099347 & 0.072082 & 0.680450 \\
50  & 0.147783 & 0.099100 & 0.071926 & 0.681191 \\
75  & 0.147335 & 0.098774 & 0.071718 & 0.682173 \\
100 & 0.147112 & 0.098611 & 0.071615 & 0.682662 \\
150 & 0.146890 & 0.098449 & 0.071511 & 0.683150 \\
200 & 0.146779 & 0.098368 & 0.071460 & 0.683394 \\
250 & 0.146712 & 0.098320 & 0.071429 & 0.683540 \\
\hline
\(\infty\) & 0.146447 & 0.098126 & 0.071305 & 0.684122 \\
\hline
\end{tabular}
\caption{Convergence of the empirical proportions \(f_n/g_n,a_n/g_n,b_n/g_n,t_n/g_n\) to the
limit (convergence) constants \(\pbot ,p_a,p_b,\ptop \). The limiting values are shown to six
decimal places (their exact nested--radical forms are given in
Theorem~\ref{thm:godel4-exact-proportions}).}
\label{tab:g4-proportions}
\end{table}


\oneLinePart{\texorpdfstring{Gödel $m$--Valued Implication}{Gödel m--Valued Implication}}
\label{sec:godel-m}

\section{\texorpdfstring{Recurrences, generating functions, and the $\Phi$--calculus}{Recurrences, generating functions, and the Phi-calculus}}

Fix an integer \(m\ge2\) and let
\[
V_m=\{v_0<v_1<\cdots<v_{m-1}\},\qquad v_0=0,\quad v_{m-1}=1,
\]
be the Gödel chain of \(m\) truth values. The Gödel implication \(\Rightarrow_G:V_m\times V_m\to V_m\) is
\[
v_p\Rightarrow_G v_q=
\begin{cases}
v_{m-1},& p\le q,\\
v_q,& p>q.
\end{cases}
\]

Let \(\mathcal{B}_n\) be the set of full bracketings of
\[
p_1\Rightarrow p_2\Rightarrow\cdots\Rightarrow p_n,
\]
so \(|\mathcal{B}_n|=C_{n-1}\). A valuation is a map \(\nu:\{p_1,\dots,p_n\}\to V_m\), hence there are \(m^n\) valuations.

For each \(j\in\{0,1,\dots,m-1\}\) and \(n\ge1\), let \(g_n^{(j)}\) denote the total number of occurrences of the value \(v_j\) in the output columns of the truth tables of all bracketings in \(\mathcal{B}_n\), taken over all valuations. Define the total
\[
g_n:=\sum_{j=0}^{m-1} g_n^{(j)}.
\]
Since there are \(C_{n-1}\) bracketings and \(m^n\) valuations, we have
\[
g_n=m^n C_{n-1}.
\]

\subsection*{Catalan recurrences}
For \(n\ge2\), each bracketing splits uniquely into a left part on \(i\) variables and a right part on \(n-i\) variables, \(1\le i\le n-1\). Consequently,
\[
g_n=\sum_{i=1}^{n-1} g_i g_{n-i}.
\]

For each \(j\le m-2\), the value \(v_j\) can arise only when the right subformula evaluates to \(v_j\) and the left subformula evaluates to a \emph{strictly larger} truth value. Hence, for \(n\ge2\) and \(j=0,1,\dots,m-2\),
\begin{equation}\label{eq:godel-m-rec}
g_n^{(j)}=\sum_{i=1}^{n-1}\Big(\sum_{p=j+1}^{m-1} g_i^{(p)}\Big)\, g_{n-i}^{(j)}.
\end{equation}
The top value is determined by complement:
\begin{equation}\label{eq:godel-m-top}
g_n^{(m-1)}=g_n-\sum_{j=0}^{m-2} g_n^{(j)}.
\end{equation}
Initial conditions are
\[
g_1^{(j)}=1\quad (j=0,1,\dots,m-1),\qquad g_1=m.
\]

\subsection*{Generating functions}
Define the component generating functions
\[
G_j(x):=\sum_{n\ge1} g_n^{(j)}x^n\qquad (j=0,1,\dots,m-1),
\]
and the total
\[
G(x):=\sum_{n\ge1} g_n x^n=\sum_{j=0}^{m-1}G_j(x).
\]
Since \(g_n=m^n C_{n-1}\), we have the closed form
\begin{equation}\label{eq:godel-m-G}
G(x)=\sum_{n\ge1}m^nC_{n-1}x^n=\frac{1-\sqrt{1-4mx}}{2}.
\end{equation}

Multiplying \eqref{eq:godel-m-rec} by \(x^n\) and summing over \(n\ge2\) gives, for each \(j=0,1,\dots,m-2\),
\begin{equation}\label{eq:godel-m-Gj}
G_j(x)=x+H_{j+1}(x)\,G_j(x),
\qquad
H_{j+1}(x):=\sum_{p=j+1}^{m-1}G_p(x).
\end{equation}
Equivalently,
\[
G_j(x)=\frac{x}{1-H_{j+1}(x)}\qquad (0\le j\le m-2),
\qquad
G_{m-1}(x)=G(x)-\sum_{j=0}^{m-2}G_j(x).
\]

\subsection*{The $\Phi$--calculus (finite depth \(m-1\))}
Define the tail sums
\[
H_k(x):=\sum_{j=k}^{m-1}G_j(x)\qquad (k=0,1,\dots,m-1),
\]
so that \(H_0(x)=G(x)\) and \(H_{m-1}(x)=G_{m-1}(x)\). We use the map
\[
\Phi_x(u):=u+\frac{x}{1-u},
\]
defined earlier in \eqref{eq:Phi-def}.

\begin{proposition}[Tail--sum iteration]\label{prop:godel-m-phi}
For each \(k=0,1,\dots,m-2\),
\[
H_k(x)=\Phi_x\!\big(H_{k+1}(x)\big).
\]
In particular, writing \(U(x):=H_{m-1}(x)=G_{m-1}(x)\),
\[
\,G(x)=\Phi_x^{\,m-1}\!\big(U(x)\big).\,
\]
\end{proposition}

\begin{proof}
From \eqref{eq:godel-m-Gj} we have \((1-H_{k+1})G_k=x\), hence \(G_k=\frac{x}{1-H_{k+1}}\).
Then
\[
H_k=G_k+H_{k+1}=H_{k+1}+\frac{x}{1-H_{k+1}}=\Phi_x(H_{k+1}).
\]
Iterating from \(k=m-2\) down to \(k=0\) gives \(H_0=\Phi_x^{m-1}(H_{m-1})\), i.e.
\(G=\Phi_x^{m-1}(U)\).
\end{proof}

\begin{proposition}[Explicit inverse and nested radicals]\label{prop:godel-m-psi}
Viewed as a map on $x\mathbb{C}[[x]]$, \(\Phi_x\) is invertible. Its inverse branch is
\[
\Psi_x(v):=\frac{1+v-\sqrt{(1-v)^2+4x}}{2},
\]
where the square root is the formal branch with constant term $1$. Hence
$v(x)\in x\mathbb{C}[[x]]$ implies $\Psi_x(v(x))\in x\mathbb{C}[[x]]$, and
\(\Phi_x(\Psi_x(v))=v\). Consequently,
\[
\,U(x)=G_{m-1}(x)=\Psi_x^{\,m-1}\!\big(G(x)\big),\qquad
G(x)=\frac{1-\sqrt{1-4mx}}{2}.\,
\]
Moreover, once \(U\) is known, all tail sums are recovered by
\[
H_{m-1}=U,\qquad H_k=\Phi_x(H_{k+1})\ (k=m-2,\dots,0),
\]
and the components by
\[
G_k=H_k-H_{k+1}\quad (k=0,\dots,m-2),\qquad G_{m-1}=H_{m-1}.
\]
\end{proposition}

\begin{proof}
The identity \(v=\Phi_x(u)=u+\frac{x}{1-u}\) is equivalent to
\(u^2-(1+v)u+(v-x)=0\), hence
\[
u=\frac{1+v\pm\sqrt{(1-v)^2+4x}}{2}.
\]
For $v(x)\in x\mathbb{C}[[x]]$, requiring $u(x)$ to have zero constant term selects the minus sign
and the square-root branch with constant term $1$, giving \(u=\Psi_x(v)\). Applying \(\Psi_x\)
successively \(m-1\) times to \(G=\Phi_x^{m-1}(U)\) gives \(U=\Psi_x^{m-1}(G)\).
The recovery formulas follow from the definitions.
\end{proof}

\section{Asymptotics}

We now determine the limiting fraction of outputs equal to \(1=v_{m-1}\) for fixed $m$ and then
study its dependence on \(m\). Fix \(m\ge2\) and let \(v_0<\cdots<v_{m-1}=1\) be the Gödel chain.  Let
\(g_n=m^nC_{n-1}\) be the total number of truth-table entries across all Gödel--\(m\) truth
tables for fully bracketed implications on \(n\) variables.  For \(j\in\{0,1,\dots,m-1\}\)
define
\[
g_n^{(j)}:=\#\{\text{output entries equal to }v_j\},
\]
and call
\[
\ptop(m):=\lim_{n\to\infty}\frac{g_n^{(m-1)}}{g_n}
\]
the \emph{convergence constant} (limiting proportion) of outputs equal to \(1\).
Since
\[
g_n\sim \frac{(4m)^n}{4\sqrt{\pi}\,n^{3/2}}\qquad (n\to\infty),
\]
the corresponding asymptotic coefficient is \(c_1(m)=\ptop(m)/4\) in
\[
g_n^{(m-1)}\sim \frac{\ptop(m)}{4}\cdot \frac{(4m)^n}{\sqrt{\pi}\,n^{3/2}}
\qquad (n\to\infty).
\]

\medskip

\noindent
For exact computation of \(\ptop(m)\) via the \(\Psi\)--chain, we set \(r=\frac{1}{4m}\), so \(4r=\frac{1}{m}\), and note that
\[
G(x)=\frac{1-\sqrt{1-4mx}}{2}
\quad\Longrightarrow\quad
G(r)=\frac12.
\]
Recall the inverse map
\[
\Psi_x(v)=\frac{1+v-\sqrt{(1-v)^2+4x}}{2}.
\]
Define the critical iterates
\[
w_0:=G(r)=\frac12,\qquad
w_{k+1}:=\Psi_r(w_k)\quad (k=0,1,\dots,m-2).
\]
Then
\[
\partial_v\Psi_x(v)=\frac12\left(1+\frac{1-v}{\sqrt{(1-v)^2+4x}}\right),
\]
and at \(x=r\) this becomes
\[
\partial_v\Psi_r(v)=\frac12\left(1+\frac{1-v}{\sqrt{(1-v)^2+\frac{1}{m}}}\right).
\]

\begin{proposition}[Limit proportion of $1$'s]\label{prop:p1m-product}
The convergence constant \(\ptop(m)=\lim_{n\to\infty} g_n^{(m-1)}/g_n\) exists and is given by
\[
\;
\ptop(m)=\prod_{k=0}^{m-2}\partial_v\Psi_r(w_k)
=\prod_{k=0}^{m-2}
\frac12\left(1+\frac{1-w_k}{\sqrt{(1-w_k)^2+\frac{1}{m}}}\right).
\;
\]
\end{proposition}

\begin{proof}
Let $U(x)=G_{m-1}(x)$. Proposition~\ref{prop:godel-m-psi} gives
\[
U(x)=\Psi_x^{\,m-1}(G(x)).
\]
At $r=1/(4m)$,
\[
G(x)=\frac12-\frac12\sqrt{1-\frac{x}{r}}.
\]
Lemma~\ref{lem:delta-domain-godel} supplies a common $\Delta$--domain for the inverse chain, so Theorem~\ref{thm:blackbox-sqrt} applies with initial square-root coefficient $1/2$. Hence
\[
U(x)=U(r)-\kappa_{m-1}\sqrt{1-\frac{x}{r}}
+O\!\left(1-\frac{x}{r}\right),
\]
where
\[
\kappa_{m-1}=\frac12\prod_{k=0}^{m-2}\partial_v\Psi_r(w_k).
\]
Square-root transfer therefore gives
\[
g_n\sim\frac{r^{-n}}{4\sqrt\pi\,n^{3/2}},
\qquad
g_n^{(m-1)}\sim\frac{\kappa_{m-1}r^{-n}}{2\sqrt\pi\,n^{3/2}}.
\]
Taking the ratio yields
\[
\ptop(m)=2\kappa_{m-1}
=\prod_{k=0}^{m-2}\partial_v\Psi_r(w_k),
\]
and substituting the explicit derivative formula gives the displayed product.
\end{proof}

\FloatBarrier
\noindent
\textbf{Numerics.}
Table~\ref{tab:p1m} lists \(\ptop(m)\) for several values of \(m\). The displayed values decrease
towards the limiting value identified in Theorem~\ref{thm:p1-large-m}.
\begin{table}[H]
\centering
\tiny
\renewcommand{\arraystretch}{1.05}
\begin{tabular}{r|cccccccc}
\hline
\(m\) & 2 & 3 & 4 & 5 & 6 & 8 & 10 & 20 \\
\hline
\(\ptop(m)\) &
0.7886751346 &
0.7190479224 &
0.6841222107 &
0.6630699362 &
0.6489761394 &
0.6312738082 &
0.6205979134 &
0.5990962874 \\
\hline
\end{tabular}
\caption{The convergence constant \(\ptop(m)=\lim_{n\to\infty} g_n^{(m-1)}/g_n\), i.e.\ the limiting
proportion of outputs equal to \(1\), for Gödel--\(m\) as \(m\) increases.}
\label{tab:p1m}
\end{table}

\begin{table}[H]
\centering
\small
\renewcommand{\arraystretch}{1.05}
\begin{tabular}{r|ccccc}
\hline
\(m\) & 100 & 200 & 1000 & 10000 & \(1/\sqrt{3}\) \\
\hline
\(\ptop(m)\) &
0.5817218321 &
0.5795375421 &
0.5777879655 &
0.5773940443 &
0.5773502692 \\
\hline
\end{tabular}
\caption{Large-\(m\) behaviour of \(\ptop(m)\), showing convergence to \(1/\sqrt{3}\).}
\label{tab:p1m-large}
\end{table}
\FloatBarrier

\noindent
\textbf{Large-\(m\) limit.}
The \(\Psi_r\)-iteration admits a continuum approximation when \(m\) is large (since \(r=1/(4m)\) is small
but the depth is \(m-1\)). This gives the large--\(m\) limit of \(\ptop(m)\).

\begin{theorem}[As $m\to\infty$]\label{thm:p1-large-m}
\label{thm:p1limit}
Let \(r=\frac1{4m}\) and let \((w_k)_{k\ge0}\) be defined by \(w_{k+1}=\Psi_r(w_k)\) with
\(w_0=\frac12\). Define
\[
\ptop(m):=\prod_{k=0}^{m-2}\partial_v\Psi_r(w_k).
\]
Then
\[
\;\;\ptop(m)\longrightarrow \frac{1}{\sqrt{3}}.\;\;
\]
\end{theorem}

\begin{proof}
Write \(u_k:=1-w_k\). In our situation \(\Psi_r\) admits the closed form
\[
\Psi_r(w)=\frac{1+w-\sqrt{(1-w)^2+\frac1m}}{2},
\qquad\text{so that}\qquad
u_{k+1}=\frac{u_k+\sqrt{u_k^2+\frac1m}}{2},
\]
and
\[
\partial_v\Psi_r(w)
=\frac12\left(1+\frac{u}{\sqrt{u^2+\frac1m}}\right),
\qquad u=1-w.
\]
Since
\[
u_{k+1}-u_k
=\frac{\sqrt{u_k^2+\frac1m}-u_k}{2}
=\frac{1/m}{2\bigl(\sqrt{u_k^2+\frac1m}+u_k\bigr)}
\ge0,
\]
the sequence \((u_k)\) is increasing, and hence \(u_k\ge u_0=\frac12\). Therefore
\[
0\le u_{k+1}-u_k
\le \frac{1/m}{4u_k}
\le \frac{1}{2m}.
\]
Summing this bound gives
\[
\frac12\le u_k\le \frac12+\frac{k}{2m}\le 1
\qquad (0\le k\le m).
\]

\medskip
\noindent
\emph{Uniform expansion of \(\log(\partial_v\Psi_r(w_k))\).}
Let \(\varepsilon:=\frac{1}{m u^2}\). For \(u\in[1/2,1]\) we have \(u^2\ge 1/4\), hence
\[
0\le \varepsilon=\frac{1}{m u^2}\le \frac{4}{m},
\]
so \(\varepsilon\to0\) uniformly as \(m\to\infty\). Using
\[
\frac{u}{\sqrt{u^2+\frac1m}}=\frac{1}{\sqrt{1+\varepsilon}}
=1-\frac{\varepsilon}{2}+O(\varepsilon^2),
\]
we obtain, uniformly for \(u\in[1/2,1]\),
\begin{equation}
\label{eq:derivexp_rewrite}
\partial_v\Psi_r(1-u)
=1-\frac{1}{4m u^2}+O\!\left(\frac{1}{m^2}\right).
\end{equation}
Set
\[
x_k:=\frac{1}{4m u_k^2},\qquad \delta_k:=\partial_v\Psi_r(w_k)-\bigl(1-x_k\bigr).
\]
Then \(0\le x_k\le 1/m\) and, by \eqref{eq:derivexp_rewrite}, \(|\delta_k|\le C/m^2\) uniformly for \(k\le m\).
In particular \(x_k-\delta_k=O(1/m)\) uniformly, hence for \(m\) large enough \(\partial_v\Psi_r(w_k)\in(1/2,1)\)
uniformly and we may use \(\log(1-y)=-y+O(y^2)\) uniformly:
\[
\log\bigl(\partial_v\Psi_r(w_k)\bigr)
=\log\bigl(1-(x_k-\delta_k)\bigr)
=-(x_k-\delta_k)+O\bigl((x_k-\delta_k)^2\bigr).
\]
Summing from \(k=0\) to \(m-2\) and using
\[
\sum_{k\le m}x_k^2\le m\cdot (1/m^2)=O(1/m),\qquad
\sum_{k\le m}|\delta_k|=O(1/m),
\]
we obtain
\begin{equation}
\label{eq:logp1_rewrite}
\log \ptop(m)
=\sum_{k=0}^{m-2}\log\bigl(\partial_v\Psi_r(w_k)\bigr)
=-\sum_{k=0}^{m-2}\frac{1}{4m u_k^2}+o(1).
\end{equation}

\medskip
\noindent
\emph{Euler approximation.}
Let \(h:=1/m\) and \(t_k:=kh\). Let \(u(t)\) be the solution of
\[
u'(t)=\frac{1}{4u(t)},\qquad u(0)=\frac12,
\]
so \(u(t)^2=\frac14+\frac{t}{2}\) and \(u(t)\in[1/2,1]\) for \(t\in[0,1]\).
From the exact recurrence,
\[
u_{k+1}-u_k=\frac{\sqrt{u_k^2+h}-u_k}{2}
=\frac{h}{2\bigl(\sqrt{u_k^2+h}+u_k\bigr)}
=\frac{h}{4u_k}+O(h^2),
\]
uniformly for \(u_k\in[1/2,1]\). Thus
\[
u_{k+1}=u_k+h f(u_k)+\zeta_k,\qquad f(u)=\frac{1}{4u},\qquad |\zeta_k|\le C_1 h^2.
\]
By Taylor's theorem applied to the ODE solution,
\[
u(t_{k+1})=u(t_k)+h f\bigl(u(t_k)\bigr)+\eta_k,\qquad |\eta_k|\le C_2 h^2,
\]
uniformly for \(t_k\in[0,1]\). Setting \(e_k:=u_k-u(t_k)\) and subtracting gives
\[
e_{k+1}=e_k+h\bigl(f(u_k)-f(u(t_k))\bigr)+(\zeta_k-\eta_k).
\]
Since \(|f'(u)|=1/(4u^2)\le 1\) on \([1/2,1]\), \(f\) is Lipschitz with constant \(L=1\), hence
\[
|e_{k+1}|\le (1+Lh)|e_k|+C_3 h^2.
\]
Iterating and using \(e_0=0\) yields \(|e_k|\le C_4 h\) uniformly for \(0\le k\le m\), i.e.
\[
\max_{0\le k\le m}|u_k-u(t_k)|=O(h)=O(1/m).
\]

\medskip
\noindent
\emph{Riemann-sum limit.}
Since \(t\mapsto 1/u(t)^2\) is continuous on \([0,1]\) and bounded, and \(u_k=u(t_k)+O(1/m)\) uniformly, we have
\[
\frac{1}{m}\sum_{k=0}^{m-2}\frac{1}{u_k^2}
=\frac{1}{m}\sum_{k=0}^{m-2}\frac{1}{u(t_k)^2}+o(1)
\longrightarrow \int_0^1\frac{dt}{u(t)^2}.
\]
Combining this with \eqref{eq:logp1_rewrite} gives
\[
\log \ptop(m)\longrightarrow -\int_0^1\frac{dt}{4u(t)^2}.
\]
Using \(u(t)^2=\frac14+\frac{t}{2}\),
\[
\int_0^1\frac{dt}{4u(t)^2}
=\int_0^1\frac{dt}{1+2t}
=\frac12\log 3.
\]
Hence \(\log \ptop(m)\to -\frac12\log 3\), i.e.
\[
\ptop(m)\to e^{-(1/2)\log 3}=\frac{1}{\sqrt3},
\]
as claimed.
\end{proof}

\begin{proposition}[Limit proportion of false rows]\label{prop:p0m}
Let \(g_n^{(0)}\) denote the total number of occurrences of the bottom truth value \(v_0=0\)
across all Gödel--\(m\) truth tables for fully bracketed implications on \(n\) variables, and set
\[
\pbot(m):=\lim_{n\to\infty}\frac{g_n^{(0)}}{g_n},\qquad g_n=m^nC_{n-1}.
\]
Then
\[
\;
\pbot(m)=\frac12\left(1-\sqrt{\frac{m}{m+4}}\right).
\;
\]
In particular \(\pbot(m)\sim \frac{1}{m}\) as \(m\to\infty\), so the proportion of false outputs
vanishes in the large--\(m\) limit.
\end{proposition}

\begin{proof}
Let \(H_k(x)=\sum_{j=k}^{m-1}G_j(x)\) be the tail sums, so \(H_0(x)=G(x)\) and \(G_0(x)=H_0(x)-H_1(x)\).
Since \(H_0=\Phi_x(H_1)\) with \(\Phi_x(u)=u+\frac{x}{1-u}\) and \(\Psi_x\) is the inverse branch,
we have \(H_1(x)=\Psi_x(G(x))\). Hence
\[
G_0(x)=G(x)-\Psi_x(G(x)).
\]
At the dominant singularity \(r=\frac{1}{4m}\) we have \(G(r)=\frac12\). By Lemma~\ref{lem:delta-domain-godel}, the inverse step \(H_1=\Psi_x(H_0)\) is \(\Delta\)--analytic at \(r\), and Theorem~\ref{thm:blackbox-sqrt} shows that its square-root coefficient is multiplied by \(\partial_v\Psi_r(1/2)\). Therefore the coefficient ratio of \(H_1\) to \(H_0\) tends to \(\partial_v\Psi_r(1/2)\), and since \(G_0=H_0-H_1\),
\[
\pbot(m)=1-\partial_v\Psi_r\!\left(\tfrac12\right).
\]
Using \(\partial_v\Psi_r(v)=\frac12\left(1+\frac{1-v}{\sqrt{(1-v)^2+\frac{1}{m}}}\right)\) gives
\[
\partial_v\Psi_r\!\left(\tfrac12\right)=\frac12\left(1+\frac{\tfrac12}{\sqrt{\tfrac14+\tfrac1m}}\right)
=\frac12\left(1+\sqrt{\frac{m}{m+4}}\right),
\]
which gives the claimed formula.
\end{proof}

\medskip
\noindent
The large--\(m\) behaviour of the tail sums \(q_k(m)\) is developed systematically in
Part~III; see Theorem~\ref{thm:macro-cut}. We record here only the fixed--\(m\) recursion
for the individual limiting proportions.

\textbf{Remark (all levels \(p_j(m)\)).}
Theorem~\ref{thm:blackbox-sqrt}, together with Lemma~\ref{lem:delta-domain-godel}, computes \emph{all} convergence constants
\(p_j(m)=\lim_{n\to\infty} g_n^{(j)}/g_n\).  If \(H_k=\sum_{j=k}^{m-1}G_j\) are the tail sums,
then at the critical point \(r=1/(4m)\) one has \(H_k=\Phi_r(H_{k+1})\) with
\(\Phi_x(u)=u+\frac{x}{1-u}\).  Writing \(q_k(m):=\lim_{n\to\infty} [x^n]H_k(x)/[x^n]G(x)\),
one gets the recursion
\[
q_{m-1}(m)=\ptop(m),\qquad
q_k(m)=\big(\partial_u\Phi_r(H_{k+1}(r))\big)\,q_{k+1}(m),
\quad
p_k(m)=q_k(m)-q_{k+1}(m),
\]
where \(\partial_u\Phi_r(u)=1+\frac{r}{(1-u)^2}\).

\paragraph{How the mass spreads across truth values as \(m\) increases.}
For fixed \(m\ge2\) and each \(j\in\{0,1,\dots,m-1\}\), define the \emph{convergence constant}
(limit proportion)
\[
p_j(m)\;:=\;\lim_{n\to\infty}\frac{g_n^{(j)}}{g_n},
\qquad
g_n=m^nC_{n-1}.
\]
Thus \(p_j(m)\) is the asymptotic fraction of outputs equal to \(v_j\) as \(n\to\infty\).
In particular, \(\ptop(m)\) is the limiting proportion of \(1=v_{m-1}\).
These constants satisfy \(\sum_{j=0}^{m-1}p_j(m)=1\), and they determine the asymptotic
coefficients via
\[
g_n^{(j)}\sim \frac{p_j(m)}{4}\cdot \frac{(4m)^n}{\sqrt{\pi}\,n^{3/2}}
\qquad (n\to\infty).
\]

Tables~\ref{tab:pvector-m5} and \ref{tab:pvector-m10} list the full limiting vectors
\((\pbot(m),p_1(m),\dots,p_{m-2}(m), \ptop(m))\) for \(m=5\) and \(m=10\), illustrating how the mass
is distributed among intermediate truth values as \(m\) grows.

\begin{table}[H]
\centering
\small
\renewcommand{\arraystretch}{1.05}
\begin{tabular}{c|ccccc}
\hline
\(j\) (value \(v_j\)) & 0 & 1 & 2 & 3 & 4 \\
\hline
\(p_j(5)\) & 0.127322 & 0.089600 & 0.067242 & 0.052766 & 0.663070 \\
\hline
\end{tabular}
\caption{Limit proportions \(p_j(5)=\lim_{n\to\infty} g_n^{(j)}/g_n\) for Gödel--\(5\).  The
top entry \(p_4(5)\) is the limiting proportion of \(1=v_4\).}
\label{tab:pvector-m5}
\end{table}

\begin{table}[H]
\centering
\tiny
\renewcommand{\arraystretch}{1.05}
\begin{tabular}{c|cccccccccc}
\hline
\(j\) (value \(v_j\)) & 0 & 1 & 2 & 3 & 4 & 5 & 6 & 7 & 8 & 9 \\
\hline
\(p_j(10)\) &
0.077423 & 0.062152 & 0.051275 & 0.043214 & 0.037050 &
0.032213 & 0.028338 & 0.025177 & 0.022560 & 0.620598 \\
\hline
\end{tabular}
\caption{Limit proportions \(p_j(10)=\lim_{n\to\infty} g_n^{(j)}/g_n\) for Gödel--\(10\).  The
top entry \(p_9(10)\) is the limiting proportion of \(1=v_9\).}
\label{tab:pvector-m10}
\end{table}

\newpage

\oneLinePart{\texorpdfstring{Continuum Limit as $m\to\infty$}{Continuum Limit as m tends to infinity}}\label{sec:limit-law-shape}

This part contains the main qualitative result of the paper. For fixed $m$, Part~II produces a discrete limiting distribution on the $m$ truth levels. We now let the chain itself become dense and show that, after rescaling the rank, these discrete laws converge to an explicit probability measure on $[0,1]$. The key point is that the same derivative product governing the finite-chain tail masses becomes a Riemann sum along the Euler limit of the critical $\Psi$-iteration.

For each $m\ge2$, let
\[
p_k(m):=\lim_{n\to\infty}\frac{g_n^{(k)}}{g_n}\qquad (k=0,1,\dots,m-1)
\]
denote the limiting proportion of output entries equal to $v_k$ in the Gödel--$m$ truth tables.
Equivalently, $(p_k(m))_{k=0}^{m-1}$ is a probability distribution on $\{0,1,\dots,m-1\}$.
Define the tail (survival) sums
\[
q_k(m):=\sum_{j=k}^{m-1}p_j(m)=\Pr(K_m\ge k),
\]
where $K_m$ is a random variable with $\Pr(K_m=k)=p_k(m)$.

\section{Macroscopic cuts}

\begin{theorem}[Macroscopic cut for nonvanishing limits]\label{thm:macro-cut}
For each \(m\ge2\) and \(k\in\{0,1,\dots,m-1\}\), one has
\[
q_k(m)=\lim_{n\to\infty}\frac{[x^n]H_k(x)}{[x^n]G(x)}
=\lim_{n\to\infty}\frac{\#\{\text{outputs}\ge v_k\}}{\#\{\text{all outputs}\}},
\]
so that \(q_{m-1}(m)=p_{m-1}(m)\) and
\(p_k(m)=q_k(m)-q_{k+1}(m)\) for \(0\le k\le m-2\).

Let \(k_m\in\{0,1,\dots,m-1\}\) be any sequence satisfying
\[
\frac{k_m}{m-1}\longrightarrow t\in[0,1].
\]
Then, as \(m\to\infty\),
\[
\qquad
q_{k_m}(m)\longrightarrow \frac{1}{\sqrt{1+2t}}.
\qquad
\]
In particular, taking \(k_m=\lfloor t(m-1)\rfloor\) gives
\[
q_{\lfloor t(m-1)\rfloor}(m)\longrightarrow \frac{1}{\sqrt{1+2t}}.
\]
Consequently, for every fixed level index \(j\) one has \(p_j(m)\to0\) as \(m\to\infty\).
\end{theorem}

\begin{proof}
Write \(k:=k_m\). Let \(r=\frac1{4m}\), set \(w_0=\frac12\), and define
\[
w_{j+1}:=\Psi_r(w_j),\qquad u_j:=1-w_j.
\]
Let \(h:=1/m\) and \(t_j:=jh\). As in the proof of
Theorem~\ref{thm:p1limit}, the recursion for \((u_j)\) is an Euler scheme for
\[
u'(s)=\frac{1}{4u(s)},\qquad u(0)=\frac12,
\]
whose solution satisfies
\[
u(s)^2=\frac14+\frac{s}{2}\qquad(0\le s\le1).
\]
The same Euler estimate gives
\[
\max_{0\le j\le m}|u_j-u(t_j)|=O(1/m).
\]
In particular \(u_j\in[1/2,1]\) for \(0\le j\le m\), and uniformly in this range
\[
\partial_v\Psi_r(w_j)
=1-\frac{1}{4m\,u_j^2}+O\!\left(\frac1{m^2}\right).
\]

By Lemma~\ref{lem:delta-domain-godel}, the first \(k\) inverse iterates are \(\Delta\)--analytic at \(r\). Repeated application of Theorem~\ref{thm:blackbox-sqrt} therefore multiplies the square-root coefficient at step \(j\) by \(\partial_v\Psi_r(w_j)\), giving
\[
q_k(m)=\prod_{j=0}^{k-1}\partial_v\Psi_r(w_j),
\]
with the empty product interpreted as \(1\) when \(k=0\). Hence, uniformly for \(k\le m\),
\[
\log q_k(m)
=-\frac{1}{4m}\sum_{j=0}^{k-1}\frac{1}{u_j^2}+o(1).
\]
Because \(k/(m-1)\to t\), also \(k/m\to t\). Using the uniform Euler estimate,
\[
\frac1m\sum_{j=0}^{k-1}\frac{1}{u_j^2}
=\frac1m\sum_{j=0}^{k-1}\frac{1}{u(t_j)^2}+o(1)
\longrightarrow \int_0^t\frac{ds}{u(s)^2}.
\]
Therefore
\[
\log q_k(m)
\longrightarrow
-\int_0^t\frac{ds}{4u(s)^2}
=-\int_0^t\frac{ds}{1+2s}
=-\frac12\log(1+2t),
\]
and so \(q_k(m)\to(1+2t)^{-1/2}\).

Finally, if \(j\) is fixed, then both \(j/(m-1)\to0\) and
\((j+1)/(m-1)\to0\). Hence
\[
q_j(m)\to1,\qquad q_{j+1}(m)\to1,
\]
and therefore \(p_j(m)=q_j(m)-q_{j+1}(m)\to0\).
\end{proof}

\section{Limit law and limit shape}
Define the rescaled level
\[
X_m := \frac{K_m}{m-1}\in[0,1].
\]
Then, for $k=0,1,\dots,m-1$,
\[
\Pr\!\left(X_m=\frac{k}{m-1}\right)=\Pr(K_m=k)=p_k(m).
\]
Equivalently, let
\[
\mu_m := \sum_{k=0}^{m-1} p_k(m)\,\delta_{k/(m-1)}
\]
be the law of $X_m$ as a probability measure on $[0,1]$.

\newpage

\begin{theorem}[Limit law / weak limit of the scaled output level]\label{thm:limit-law-shape}
For \(t\in[0,1)\), let
\[
k_m(t):=\left\lceil t(m-1)\right\rceil .
\]
Then
\begin{equation}\label{eq:survival-Xm}
\Pr(X_m\ge t)
=\sum_{j=k_m(t)}^{m-1}p_j(m)
=q_{k_m(t)}(m)
\longrightarrow \frac1{\sqrt{1+2t}}.
\end{equation}
Moreover,
\begin{equation}\label{eq:top-mass}
\Pr(X_m=1)=p_{m-1}(m)=q_{m-1}(m)\longrightarrow \frac1{\sqrt3}.
\end{equation}
Consequently \(X_m\Rightarrow X\) (equivalently \(\mu_m\Rightarrow\mu\)), where
\begin{equation}\label{eq:limit-law}
\mu(dx)=\frac{1}{(1+2x)^{3/2}}\mathbf 1_{[0,1)}(x)\,dx
+\frac1{\sqrt3}\,\delta_1(dx).
\end{equation}
Equivalently,
\[
\Pr(X\ge t)=\frac1{\sqrt{1+2t}}\quad(0\le t<1),
\qquad
\Pr(X=1)=\frac1{\sqrt3}.
\]
\end{theorem}

\begin{proof}
For fixed \(t\in[0,1)\),
\[
\frac{k_m(t)}{m-1}\longrightarrow t,
\]
so \eqref{eq:survival-Xm} follows directly from
Theorem~\ref{thm:macro-cut}. Taking \(k_m=m-1\) in the same theorem gives
\eqref{eq:top-mass}.

To prove weak convergence, let
\[
D_m(t):=\Pr(X_m\le t).
\]
For \(0\le t<1\), put
\[
\ell_m(t):=\lfloor t(m-1)\rfloor+1.
\]
Then, exactly,
\[
D_m(t)
=1-\Pr(X_m>t)
=1-q_{\ell_m(t)}(m).
\]
Since \(\ell_m(t)/(m-1)\to t\), Theorem~\ref{thm:macro-cut} gives
\[
D_m(t)\longrightarrow
1-\frac1{\sqrt{1+2t}}
\qquad(0\le t<1).
\]
For \(t<0\) both distribution functions are \(0\), while for \(t>1\) they are \(1\).
Thus \(D_m(t)\to D_X(t)\) at every continuity point of
\[
D_X(t)=
\begin{cases}
0, & t<0,\\[2mm]
1-\dfrac1{\sqrt{1+2t}}, & 0\le t<1,\\[3mm]
1, & t\ge1.
\end{cases}
\]
Hence \(X_m\Rightarrow X\).

On \([0,1)\), the limiting distribution has density
\[
f(t)=D_X'(t)=\frac{1}{(1+2t)^{3/2}}.
\]
Its total mass there is
\[
\int_0^1\frac{dt}{(1+2t)^{3/2}}
=1-\frac1{\sqrt3},
\]
so the remaining mass \(1/\sqrt3\) is an atom at \(t=1\), in agreement with
\eqref{eq:top-mass}. This is precisely the measure in \eqref{eq:limit-law}.
\end{proof}

\begin{corollary}[A derived statistic]
With $X$ as in Theorem~\ref{thm:limit-law-shape}, one has $\mathbb{E}[X]=\sqrt{3}-1$.
\end{corollary}

\begin{proof}
By Theorem~\ref{thm:limit-law-shape} (limit law \eqref{eq:limit-law}), $X$ has density
\[
f(t)=\frac{1}{(1+2t)^{3/2}}\mathbf 1_{[0,1)}(t)
\]
and an atom of mass $1/\sqrt3$ at $t=1$. Hence
\[
\mathbb{E}[X]=\int_0^1 \frac{t}{(1+2t)^{3/2}}\,dt+\frac{1}{\sqrt3}.
\]
With the substitution $u=1+2t$ (so $t=(u-1)/2$ and $dt=du/2$), we get
\[
\int_0^1 \frac{t}{(1+2t)^{3/2}}\,dt
=\frac14\int_1^3 (u^{-1/2}-u^{-3/2})\,du
=\frac12\Bigl(\sqrt{u}+u^{-1/2}\Bigr)\Big|_{1}^{3}
=\frac{2}{\sqrt3}-1.
\]
Therefore
\[
\mathbb{E}[X]=\left(\frac{2}{\sqrt3}-1\right)+\frac{1}{\sqrt3}
=\frac{3}{\sqrt3}-1
=\sqrt3-1.
\]
\end{proof}

\begin{remark}[Equivalent weak convergence formulation]
The convergence $X_m \Rightarrow X$ in Theorem~\ref{thm:limit-law-shape} is equivalent to weak convergence of
measures $\mu_m \Rightarrow \mu$, i.e.
\[
\int_{[0,1]} \varphi\, d\mu_m
=\sum_{k=0}^{m-1} p_k(m)\,\varphi\!\left(\frac{k}{m-1}\right)
\longrightarrow
\int_{[0,1]} \varphi\, d\mu,
\]
for every bounded continuous function $\varphi:[0,1]\to\mathbb{R}$, where
\[
\mu(dt)=\frac{1}{(1+2t)^{3/2}}\mathbf 1_{[0,1)}(t)\,dt+\frac{1}{\sqrt3}\,\delta_1(dt).
\]
\end{remark}

\section{Direct Lebesgue semantics on the unit interval}\label{sec:direct-unit-interval}

The preceding limit law was obtained by first fixing a finite G\"odel chain, letting the
number of variables tend to infinity, and then letting the number of truth values tend to
infinity.  There is a natural direct continuum model in which the truth-value space is
$[0,1]$ from the outset.  We record it here because it gives an intrinsic interpretation of
Theorem~\ref{thm:limit-law-shape} and turns the limiting differential equation of the
finite-chain argument into an exact identity.

Let $\lambda$ denote Lebesgue probability measure on $[0,1]$, and retain the usual G\"odel
implication
\[
x\Rightarrow_G y=
\begin{cases}
1,&x\le y,\\
y,&x>y.
\end{cases}
\]
For a bracketing $\beta\in\mathcal B_n$, let
$E_\beta:[0,1]^n\to[0,1]$ be its evaluation map.  Define the aggregate output measure
\begin{equation}\label{eq:continuum-aggregate-measure}
\mathcal M_n(B):=
\sum_{\beta\in\mathcal B_n}
\lambda^{\otimes n}\!\left(E_\beta^{-1}(B)\right),
\qquad B\subseteq[0,1]\ \text{Borel}.
\end{equation}
Since $|\mathcal B_n|=C_{n-1}$,
\[
\mathcal M_n([0,1])=C_{n-1}.
\]
Thus
\[
\nu_n:=\frac{\mathcal M_n}{C_{n-1}}
\]
is the output law obtained by choosing a full bracketing uniformly from $\mathcal B_n$ and
choosing the $n$ input values independently according to $\lambda$.

The Catalan decomposition remains exact at the level of measures.  If
$\Gamma(u,v)=u\Rightarrow_G v$, then for $n\ge2$,
\begin{equation}\label{eq:continuum-measure-catalan}
\mathcal M_n
=
\sum_{i=1}^{n-1}
\Gamma_\#\bigl(\mathcal M_i\otimes\mathcal M_{n-i}\bigr).
\end{equation}
Indeed, a root split of sizes $i$ and $n-i$ uses disjoint variable blocks, so the left and
right evaluations are independent under product Lebesgue measure.

For $0\le t<1$, define the aggregate strict tail
\begin{equation}\label{eq:continuum-tail-An}
A_n(t):=\mathcal M_n((t,1]).
\end{equation}
The restriction of $\mathcal M_n$ to $[0,1)$ is absolutely continuous for every $n$.
This follows inductively from \eqref{eq:continuum-measure-catalan}: a non-top output equal
to $y<1$ can occur only when the right subtree has value $y$ and the left subtree has
value strictly larger than $y$.  Write $f_n(t)$ for the density on $[0,1)$.  Then
\begin{equation}\label{eq:continuum-density-tail}
-\frac{d}{dt}A_n(t)=f_n(t),
\qquad
A_1(t)=1-t,
\qquad
f_1(t)=1,
\end{equation}
and, for $n\ge2$,
\begin{equation}\label{eq:continuum-density-recurrence}
f_n(t)=
\sum_{i=1}^{n-1}A_i(t)f_{n-i}(t).
\end{equation}
Equation \eqref{eq:continuum-density-recurrence} is the measure-theoretic analogue of the
finite-chain triangular recurrence: the right subtree supplies the output value while the
left subtree must lie strictly above it.

Introduce the generating functions
\[
A(t,x):=\sum_{n\ge1}A_n(t)x^n,
\qquad
F(t,x):=\sum_{n\ge1}f_n(t)x^n,
\]
and the shifted Catalan generating function
\begin{equation}\label{eq:continuum-catalan-C}
\mathcal C(x):=\sum_{n\ge1}C_{n-1}x^n
=\frac{1-\sqrt{1-4x}}2.
\end{equation}
There are no atoms at $0$, so $A_n(0)=C_{n-1}$ and hence
\begin{equation}\label{eq:continuum-tail-boundary}
A(0,x)=\mathcal C(x).
\end{equation}
Multiplying \eqref{eq:continuum-density-recurrence} by $x^n$ and summing gives
\[
F(t,x)-x=A(t,x)F(t,x),
\]
so
\[
F(t,x)=\frac{x}{1-A(t,x)}.
\]
Together with $\partial_tA=-F$, this yields the exact continuum equation
\begin{equation}\label{eq:continuum-exact-ode}
\qquad
(1-A(t,x))\,\partial_tA(t,x)=-x,
\qquad
A(0,x)=\mathcal C(x).
\qquad
\end{equation}

\begin{proposition}[Exact continuum tail generating function]\label{prop:continuum-tail-exact}
For $0\le t<1$,
\begin{equation}\label{eq:continuum-A-closed}
A(t,x)
=
1-\sqrt{(1-\mathcal C(x))^2+2xt},
\end{equation}
where the square-root branch is chosen so that
$A(0,x)=\mathcal C(x)$.
\end{proposition}

\begin{proof}
Integrating \eqref{eq:continuum-exact-ode} with respect to $t$ gives
\[
A-\frac{A^2}{2}
=
\mathcal C-\frac{\mathcal C^2}{2}-xt.
\]
Rearranging,
\[
(1-A)^2=(1-\mathcal C)^2+2xt,
\]
and the boundary condition \eqref{eq:continuum-tail-boundary} selects
\eqref{eq:continuum-A-closed}.
\end{proof}

The relation with the Euler limit used earlier is now transparent.  At the dominant
Catalan singularity $x=1/4$, one has $\mathcal C(1/4)=1/2$.  Setting
\[
u(t):=1-A(t,1/4),
\]
Equation~\eqref{eq:continuum-exact-ode} becomes
\begin{equation}\label{eq:continuum-critical-ode}
u'(t)=\frac{1}{4u(t)},
\qquad
u(0)=\frac12,
\end{equation}
so that
\[
u(t)^2=\frac14+\frac t2.
\]
Thus the ODE that arose in Theorem~\ref{thm:macro-cut} as the continuum limit of the
critical $\Psi$-iteration is already the exact critical flow of the direct interval model.

We next recover the limiting law.  Put
$s=\sqrt{1-4x}$.  Then
$\mathcal C(x)=(1-s)/2$, and from \eqref{eq:continuum-A-closed}, for each fixed
$t\in[0,1]$,
\begin{equation}\label{eq:continuum-A-singular}
A(t,x)
=
1-\frac{\sqrt{1+2t}}2
-\frac{1}{2\sqrt{1+2t}}\sqrt{1-4x}
+O(1-4x).
\end{equation}
Since
\[
\mathcal C(x)=\frac12-\frac12\sqrt{1-4x},
\]
the square-root transfer theorem gives
\begin{equation}\label{eq:continuum-tail-limit}
\frac{A_n(t)}{C_{n-1}}
\longrightarrow
\frac1{\sqrt{1+2t}}
\qquad(0\le t<1).
\end{equation}
The left-hand side is precisely $\nu_n((t,1])$.

The top value must be treated as an atom.  Let
$a_n:=\mathcal M_n(\{1\})$.  By monotone convergence,
\[
\sum_{n\ge1}a_nx^n
=
\lim_{t\uparrow1}A(t,x)
=
1-\sqrt{1-\mathcal C(x)^2},
\]
where we used $x=\mathcal C-\mathcal C^2$.  Equivalently, taking $t=1$ in the
singular expansion \eqref{eq:continuum-A-singular} gives
\begin{equation}\label{eq:continuum-top-limit}
\frac{a_n}{C_{n-1}}\longrightarrow\frac1{\sqrt3}.
\end{equation}

\begin{theorem}[Direct interval law]\label{thm:direct-interval-law}
Let $Y_n$ have distribution $\nu_n$, the output law for a uniformly chosen bracketing on
$n$ variables with independent Lebesgue-distributed truth values in $[0,1]$.  Then
\[
Y_n\Rightarrow Y,
\]
where
\begin{equation}\label{eq:direct-interval-limit-measure}
\Pr(Y\in dt)
=
\frac{1}{(1+2t)^{3/2}}\mathbf 1_{[0,1)}(t)\,dt
+\frac1{\sqrt3}\,\delta_1(dt).
\end{equation}
Equivalently,
\[
\Pr(Y\ge t)=\frac1{\sqrt{1+2t}}
\quad(0\le t<1),
\qquad
\Pr(Y=1)=\frac1{\sqrt3}.
\]
In particular, the direct Lebesgue model has exactly the same limiting law as the
finite-chain scaling limit in Theorem~\ref{thm:limit-law-shape}.
\end{theorem}

\begin{proof}
For $t<1$, \eqref{eq:continuum-tail-limit} gives convergence of the survival functions.
The limiting distribution is continuous at every $t<1$, with density obtained by
differentiating
$1-(1+2t)^{-1/2}$, namely $(1+2t)^{-3/2}$.  Its mass on $[0,1)$ is
$1-1/\sqrt3$, and \eqref{eq:continuum-top-limit} supplies the remaining atom at $1$.
Hence the distribution functions converge at every continuity point of the measure in
\eqref{eq:direct-interval-limit-measure}, proving weak convergence.
\end{proof}

\begin{remark}[Two routes to the same continuum law]
Theorem~\ref{thm:limit-law-shape} and Theorem~\ref{thm:direct-interval-law} identify the
same probability measure by two different constructions.  The first takes the large-$n$
output law on each finite chain and then lets $m\to\infty$; the second starts directly from
$[0,1]$ with Lebesgue semantics and lets $n\to\infty$.  This agreement does not by itself
assert a uniform joint limit for arbitrary $m=m(n)\to\infty$, but it shows that the
continuum law is intrinsic to the order structure of G\"odel implication rather than an
artefact of the finite discretisation.
\end{remark}

\oneLinePart{Root--Split Refinement}

The preceding parts determine the asymptotic distribution of the output value. We finish with a local refinement at the root of the bracketing. Its role is deliberately secondary: it shows that the same generating functions also resolve the ordered pair of truth values entering the final implication, and that this finer decomposition carries no new singular scale.

\section{Pair-count factorisation}\label{sec:pair-factorisation}

Let $G_k(x)=\sum_{n\ge1}g_n^{(k)}x^n$ be the OGF for output value $v_k$ in the G\"odel chain
$V_m=\{v_0<\cdots<v_{m-1}\}$. For an ordered pair $(i,j)$, let $N_n^{i,j}$ count rows whose root split $\varphi\Rightarrow_G\psi$ satisfies $\nu(\varphi)=v_i$ and $\nu(\psi)=v_j$, with $N_1^{i,j}=0$.

\begin{theorem}[Root-pair factorisation]\label{thm:pair-sequences}
For every $(i,j)\in\{0,\ldots,m-1\}^2$,
\[
N^{i,j}(x):=\sum_{n\ge1}N_n^{i,j}x^n=G_i(x)G_j(x).
\]
For $n\ge2$, the output counts are recovered by
\[
g_n^{(j)}=\sum_{i=j+1}^{m-1}N_n^{i,j}\qquad(0\le j\le m-2),
\]
and
\[
g_n^{(m-1)}=\sum_{0\le i\le j\le m-1}N_n^{i,j}.
\]
Equivalently,
\[
G_j(x)-x=\sum_{i=j+1}^{m-1}N^{i,j}(x),
\qquad
G_{m-1}(x)-x=\sum_{0\le i\le j\le m-1}N^{i,j}(x).
\]
\end{theorem}

\begin{proof}
A bracketing of size $n\ge2$ has a unique root split with left size $t$ and right size $n-t$. Hence
\[
N_n^{i,j}=\sum_{t=1}^{n-1}g_t^{(i)}g_{n-t}^{(j)},
\]
which is exactly the coefficient convolution for $G_iG_j$. The reconstruction identities then follow from the G\"odel rule: the top value occurs when $i\le j$, while a non-top value $v_j$ occurs precisely when the right value is $v_j$ and the left value is strictly larger.
\end{proof}

Thus there are $m^2$ indexed ordered pair classes, although their numerical sequences need not be distinct. Swap symmetry gives $N^{i,j}=N^{j,i}$, so at most $\binom{m+1}{2}$ distinct pair-count OGFs occur for fixed $m$.

\begin{theorem}[Square-root asymptotics for the root-pair classes]\label{thm:pair-uniform-asympt}
For each $k$ there exist constants $\rho_k,\sigma_k>0$ such that, with $r=1/(4m)$,
\[
G_k(x)=\rho_k-\sigma_k\sqrt{1-\frac{x}{r}}+O\!\left(1-\frac{x}{r}\right)
\]
in a $\Delta$--domain at $r$. Consequently, for every ordered pair $(i,j)$,
\[
N^{i,j}(x)
=\rho_i\rho_j-(\rho_i\sigma_j+\rho_j\sigma_i)\sqrt{1-\frac{x}{r}}
+O\!\left(1-\frac{x}{r}\right),
\]
and therefore
\[
N_n^{i,j}\sim
\frac{\rho_i\sigma_j+\rho_j\sigma_i}{2\sqrt\pi}(4m)^n n^{-3/2}.
\]
\end{theorem}

\begin{proof}
Let $H_\ell=\sum_{k=\ell}^{m-1}G_k$. By Lemma~\ref{lem:delta-domain-godel}, the inverse chain
$H_{\ell+1}=\Psi_x(H_\ell)$ is $\Delta$--analytic at $r$, and Theorem~\ref{thm:blackbox-sqrt} gives
\[
H_\ell(x)=h_\ell-c_\ell\sqrt{1-\frac{x}{r}}+O\!\left(1-\frac{x}{r}\right),
\qquad
c_{\ell+1}=c_\ell\partial_v\Psi_r(h_\ell).
\]
At the critical point,
\[
h_\ell=\Phi_r(h_{\ell+1})
=h_{\ell+1}+\frac{r}{1-h_{\ell+1}},
\]
so $h_\ell>h_{\ell+1}>0$. Lemma~\ref{lem:delta-domain-godel} also gives $h_\ell\le H_0(r)=1/2$, and therefore
\[
0<\frac{1-h_\ell}{\sqrt{(1-h_\ell)^2+4r}}<1.
\]
Using
\[
\partial_v\Psi_r(h_\ell)
=\frac12\left(1+\frac{1-h_\ell}{\sqrt{(1-h_\ell)^2+4r}}\right),
\]
we obtain
\[
\frac12<\partial_v\Psi_r(h_\ell)<1,
\qquad
0<c_{\ell+1}<c_\ell.
\]
Thus, for $k\le m-2$,
\[
G_k=H_k-H_{k+1}
\]
has constant term $\rho_k=h_k-h_{k+1}>0$ and square-root coefficient $\sigma_k=c_k-c_{k+1}>0$; the top component has $\rho_{m-1}=h_{m-1}>0$ and $\sigma_{m-1}=c_{m-1}>0$. Multiplying the two component expansions gives the stated expansion for $N^{i,j}=G_iG_j$, and the square-root transfer theorem gives the coefficient asymptotic.
\end{proof}

\section{The four-valued instance}

For $V_4=\{0<a<b<1\}$ the factorisation immediately recovers the recurrences of Part~I. For example,
\[
F(x)-x=(A+B+T)F,\qquad
A(x)-x=(B+T)A,\qquad
B(x)-x=TB,
\]
while
\[
T(x)-x=F^2+FA+FB+FT+A^2+AB+AT+B^2+BT+T^2.
\]
Thus the $16$ ordered root-pair classes refine the four output sequences without introducing a new analytic mechanism. The first ten terms of the base and pair-count sequences are collected in Appendix~\ref{app:sample-pairs} for reference.

\section{Finally}

The order structure of finite G\"odel chains turns Catalan truth-table enumeration into a one-dimensional analytic iteration. For fixed $m$, this yields a common square-root singularity for every truth level and hence a universal $(4m)^n n^{-3/2}$ scale. In the four-valued case the corresponding limiting masses are explicit nested radicals.

The principal asymptotic phenomenon appears when $m$ also grows. The critical inverse iteration becomes an Euler discretisation of a simple differential equation, and the discrete truth-level distributions converge to an explicit law on $[0,1]$ with density $(1+2t)^{-3/2}$ below $1$ and an atom of mass $1/\sqrt3$ at the top. Section~\ref{sec:direct-unit-interval} shows that the same law is obtained by starting directly from the standard interval with Lebesgue product semantics. In that model the tail generating function satisfies an exact differential equation, and at the dominant Catalan singularity this equation is precisely the ODE arising from the finite-chain Euler limit. Thus the continuum law is simultaneously the macroscopic limit of the finite-chain recurrence and the intrinsic large-$n$ law of the interval model.

Several directions remain natural. One may ask for quantitative error bounds in the passage $m\to\infty$, for analogous continuum laws for other finite-chain implications, and for a finer comparison with associative-spectrum questions, where bracketings are grouped by the term operations they induce rather than aggregated by output frequencies. The root-pair factorisation also suggests studying which product sequences coincide within a fixed chain and whether such coincidences admit an algebraic classification.

\newpage
\appendix

\section{Supplementary material}

\begin{table}[h]
\centering
\[
\begin{array}{c|c|c}
p_1 & p_2 & p_1 \Rightarrow_G p_2\\
\hline
0 & 0 & 1\\
0 & a & 1\\
0 & b & 1\\
0 & 1 & 1\\
a & 0 & 0\\
a & a & 1\\
a & b & 1\\
a & 1 & 1\\
b & 0 & 0\\
b & a & a\\
b & b & 1\\
b & 1 & 1\\
1 & 0 & 0\\
1 & a & a\\
1 & b & b\\
1 & 1 & 1
\end{array}
\]
\caption{Explicit truth table for \(p_1 \Rightarrow_G p_2\) on \(V_4\).}
\label{tab:godel-n2}
\end{table}

From Table~\ref{tab:godel-n2} we obtain the distribution of truth values in the output column of the truth table.
There are three occurrences of the value zero, two occurrences of the value \(a\), one occurrence
of the value \(b\), and ten occurrences of the value one. Consequently, the total number of
entries is sixteen.

\newpage
\begin{table}[ht]
\centering
\caption{Truth tables for the two bracketings on three variables in Gödel four-valued logic.}
\label{tab:godel4-n3-sidebyside}

\begin{minipage}{0.3\textwidth}
\centering
\resizebox{0.60\linewidth}{!}{%
\begin{tabular}{cccc}
\hline
$p_1$ & $p_2$ & $p_3$ & ($p_1\Rightarrow_G p_2) \Rightarrow_G p_3$\\
\hline
0 & 0 & 0 & 0 \\
0 & 0 & a & a \\
0 & 0 & b & b \\
0 & 0 & 1 & 1 \\
0 & a & 0 & 0 \\
0 & a & a & a \\
0 & a & b & b \\
0 & a & 1 & 1 \\
0 & b & 0 & 0 \\
0 & b & a & a \\
0 & b & b & b \\
0 & b & 1 & 1 \\
0 & 1 & 0 & 0 \\
0 & 1 & a & a \\
0 & 1 & b & b \\
0 & 1 & 1 & 1 \\
a & 0 & 0 & 1 \\
a & 0 & a & 1 \\
a & 0 & b & 1 \\
a & 0 & 1 & 1 \\
a & a & 0 & 0 \\
a & a & a & a \\
a & a & b & b \\
a & a & 1 & 1 \\
a & b & 0 & 0 \\
a & b & a & a \\
a & b & b & b \\
a & b & 1 & 1 \\
a & 1 & 0 & 0 \\
a & 1 & a & a \\
a & 1 & b & b \\
a & 1 & 1 & 1 \\
b & 0 & 0 & 1 \\
b & 0 & a & 1 \\
b & 0 & b & 1 \\
b & 0 & 1 & 1 \\
b & a & 0 & 0 \\
b & a & a & 1 \\
b & a & b & 1 \\
b & a & 1 & 1 \\
b & b & 0 & 0 \\
b & b & a & a \\
b & b & b & b \\
b & b & 1 & 1 \\
b & 1 & 0 & 0 \\
b & 1 & a & a \\
b & 1 & b & b \\
b & 1 & 1 & 1 \\
1 & 0 & 0 & 1 \\
1 & 0 & a & 1 \\
1 & 0 & b & 1 \\
1 & 0 & 1 & 1 \\
1 & a & 0 & 0 \\
1 & a & a & 1 \\
1 & a & b & 1 \\
1 & a & 1 & 1 \\
1 & b & 0 & 0 \\
1 & b & a & a \\
1 & b & b & 1 \\
1 & b & 1 & 1 \\
1 & 1 & 0 & 0 \\
1 & 1 & a & a \\
1 & 1 & b & b \\
1 & 1 & 1 & 1 \\
\hline
\end{tabular}%
}
\end{minipage}
\hfill
\begin{minipage}{0.3\textwidth}
\centering
\resizebox{0.60\linewidth}{!}{%
\begin{tabular}{cccc}
\hline
$p_1$ & $p_2$ & $p_3$ & $p_1\Rightarrow_G (p_2\Rightarrow_G p_3)$\\
\hline
0 & 0 & 0 & 1 \\
0 & 0 & a & 1 \\
0 & 0 & b & 1 \\
0 & 0 & 1 & 1 \\
0 & a & 0 & 1 \\
0 & a & a & 1 \\
0 & a & b & 1 \\
0 & a & 1 & 1 \\
0 & b & 0 & 1 \\
0 & b & a & 1 \\
0 & b & b & 1 \\
0 & b & 1 & 1 \\
0 & 1 & 0 & 1 \\
0 & 1 & a & 1 \\
0 & 1 & b & 1 \\
0 & 1 & 1 & 1 \\
a & 0 & 0 & 1 \\
a & 0 & a & 1 \\
a & 0 & b & 1 \\
a & 0 & 1 & 1 \\
a & a & 0 & 0 \\
a & a & a & 1 \\
a & a & b & 1 \\
a & a & 1 & 1 \\
a & b & 0 & 0 \\
a & b & a & 1 \\
a & b & b & 1 \\
a & b & 1 & 1 \\
a & 1 & 0 & 0 \\
a & 1 & a & 1 \\
a & 1 & b & 1 \\
a & 1 & 1 & 1 \\
b & 0 & 0 & 1 \\
b & 0 & a & 1 \\
b & 0 & b & 1 \\
b & 0 & 1 & 1 \\
b & a & 0 & 0 \\
b & a & a & 1 \\
b & a & b & 1 \\
b & a & 1 & 1 \\
b & b & 0 & 0 \\
b & b & a & a \\
b & b & b & 1 \\
b & b & 1 & 1 \\
b & 1 & 0 & 0 \\
b & 1 & a & a \\
b & 1 & b & 1 \\
b & 1 & 1 & 1 \\
1 & 0 & 0 & 1 \\
1 & 0 & a & 1 \\
1 & 0 & b & 1 \\
1 & 0 & 1 & 1 \\
1 & a & 0 & 0 \\
1 & a & a & 1 \\
1 & a & b & 1 \\
1 & a & 1 & 1 \\
1 & b & 0 & 0 \\
1 & b & a & a \\
1 & b & b & 1 \\
1 & b & 1 & 1 \\
1 & 1 & 0 & 0 \\
1 & 1 & a & a \\
1 & 1 & b & b \\
1 & 1 & 1 & 1 \\
\hline
\end{tabular}%
}
\end{minipage}

\end{table}

\newpage
\section{Sample root-pair sequences}\label{app:sample-pairs}
The following table records the first ten terms of the G\"odel--4 base sequences and the root-pair sequences used as numerical checks in Part~IV.

\begin{table}[h]
\centering
\scriptsize
\setlength{\tabcolsep}{3pt}
\renewcommand{\arraystretch}{1.05}
\caption{Gödel--4 sequences (first ten terms). Base sequences and refined pair--sequences.}
\label{tab:godel4-all}
\resizebox{\textwidth}{!}{%
\begin{tabular}{l|rrrrrrrrrr}
\hline
 & 1 & 2 & 3 & 4 & 5 & 6 & 7 & 8 & 9 & 10\\
\hline
\multicolumn{11}{l}{\textbf{Base sequences}}\\
\hline
$t_n$ & 1 & 10 & 80 & 825 & 9355 & 113237 & 1431976 & 18696855 & 250122284 & 3410617188\\
$b_n$ & 1 & 1 & 11 & 101 & 1116 & 13186 & 164093 & 2116676 & 28052479 & 379673121\\
$a_n$ & 1 & 2 & 15 & 143 & 1559 & 18379 & 228175 & 2938786 & 38902987 & 526061137\\
$f_n$ & 1 & 3 & 22 & 211 & 2306 & 27230 & 338444 & 4362627 & 57788170 & 781825066\\
$g_n$ & 4 & 16 & 128 & 1280 & 14336 & 172032 & 2162688 & 28114944 & 374865920 & 5098176512\\
\hline
\multicolumn{11}{l}{\textbf{Refined false-row pair--sequences}}\\
\hline
$N^{a,0}_n$ & 0 & 1 & 5 & 43 & 443 & 5046 & 61209 & 775277 & 10134321 & 135696192\\
$N^{b,0}_n$ & 0 & 1 & 4 & 36 & 367 & 4178 & 50613 & 640554 & 8368030 & 111992206\\
$N^{1,0}_n$ & 0 & 1 & 13 & 132 & 1496 & 18006 & 226622 & 2946796 & 39285819 & 534136668\\
\hline
\multicolumn{11}{l}{\textbf{Refined $a$- and $b$-row pair--sequences}}\\
\hline
$N^{b,a}_n$ & 0 & 1 & 3 & 28 & 281 & 3185 & 38444 & 485351 & 6328829 & 84580645\\
$N^{1,a}_n$ & 0 & 1 & 12 & 115 & 1278 & 15194 & 189731 & 2453435 & 32574158 & 441480492\\
$N^{1,b}_n$ & 0 & 1 & 11 & 101 & 1116 & 13186 & 164093 & 2116676 & 28052479 & 379673121\\
\hline
\multicolumn{11}{l}{\textbf{Refined true-row pair--sequences ($u\le v$)}}\\
\hline
$N^{0,0}_n$ & 0 & 1 & 6 & 53 & 554 & 6362 & 77580 & 986253 & 12927170 & 173452334\\
$N^{0,a}_n$ & 0 & 1 & 5 & 43 & 443 & 5046 & 61209 & 775277 & 10134321 & 135696192\\
$N^{0,b}_n$ & 0 & 1 & 4 & 36 & 367 & 4178 & 50613 & 640554 & 8368030 & 111992206\\
$N^{0,1}_n$ & 0 & 1 & 13 & 132 & 1496 & 18006 & 226622 & 2946796 & 39285819 & 534136668\\
$N^{a,a}_n$ & 0 & 1 & 4 & 34 & 346 & 3915 & 47284 & 597085 & 7787516 & 104093243\\
$N^{a,b}_n$ & 0 & 1 & 3 & 28 & 281 & 3185 & 38444 & 485351 & 6328829 & 84580645\\
$N^{a,1}_n$ & 0 & 1 & 12 & 115 & 1278 & 15194 & 189731 & 2453435 & 32574158 & 441480492\\
$N^{b,b}_n$ & 0 & 1 & 2 & 23 & 224 & 2555 & 30826 & 389311 & 5077062 & 67857384\\
$N^{b,1}_n$ & 0 & 1 & 11 & 101 & 1116 & 13186 & 164093 & 2116676 & 28052479 & 379673121\\
$N^{1,1}_n$ & 0 & 1 & 20 & 260 & 3250 & 41610 & 545574 & 7306117 & 99586900 & 1377654903\\
\hline
\end{tabular}}
\vspace{2pt}
\par\footnotesize\noindent
Note that $N^{u,v}_1=0$ for all pairs $(u,v)$ since there is no binary split at size $n=1$.
\end{table}

\vspace{8em}

\begin{verse}
Onlar ki toprakta karınca,\\
suda balık,\\
havada kuş kadar çokturlar;\\
Korkak, cesur, hakim ve çocukturlar.\\
Kahreden ve yaratan ki onlardır,\\
Destanımızda yalnız onların maceraları vardır.
\end{verse}

\begin{flushleft}
Nazım Hikmet.
\end{flushleft}

\end{document}